\input ./style/arxiv-general.cfg
\documentclass[MSNbibl,number,citesort,secthm,dvips]{arxbj}
\makeatletter
   \@ifpackageloaded{graphicx}{}{\usepackage{graphicx}}
\makeatother
\usepackage{mathbh}

\volume{22}
\issue{2}
\pubyear{2016}
\firstpage{995}
\lastpage{1025}
\doi{10.3150/14-BEJ684}
\docsubty{FLA}

\makeatletter
\newcommand{\ds}{\displaystyle}
\newcommand{\LLL}{L}
\newcommand{\rrvert}{\vert}
\newcommand{\rrVert}{\Vert}
\newcommand{\llvert}{\vert}
\newcommand{\llVert}{\Vert}
\newtheorem{lem}[thm]{Lemma}
\newremark{example}[thm]{Example}
\newproclaim{definition}[thm]{Definition}
\newremark{rem}[thm]{Remark}
\makeatother

\begin{document}
\begin{frontmatter}

\title{$L_2$-variation of L\'evy driven BSDEs with non-smooth terminal
conditions}
\runtitle{$L_2$-variation of L\'evy driven BSDEs}

\begin{aug}
\author[A]{\inits{C.}\fnms{Christel}~\snm{Geiss}\corref{}\thanksref{A}\ead[label=e1]{christel.geiss@jyu.fi}}
\and
\author[B]{\inits{A.}\fnms{Alexander}~\snm{Steinicke}\thanksref{B}\ead[label=e2]{alexander.steinicke@uibk.ac.at}}
\address[A]{Department of Mathematics and Statistics, University of Jyv\"{a}skyl\"{a},  FI-40014, Finland.\\  \printead{e1}}
\address[B]{Department of Mathematics, University of Innsbruck, 6020 Innsbruck, Austria.\\ \printead{e2}}
\end{aug}

%
\received{\smonth{12} \syear{2012}}
%
\revised{\smonth{10} \syear{2014}}

%
\begin{abstract}
We consider the $\LLL_2$-regularity of solutions to backward stochastic
differential equations (BSDEs) with Lipschitz
generators driven by a Brownian motion and a Poisson random measure
associated with a L\'evy process $(X_t)_{t \in[0,T]}$. The terminal
condition may be a Borel function of finitely many increments
of the L\'evy process which is not necessarily Lipschitz but only
satisfies a fractional smoothness condition.
The results are obtained by investigating how the special structure
appearing in the chaos expansion of the terminal condition
is inherited by the solution to the BSDE.
\end{abstract}

%
\end{frontmatter}

\section{Introduction}\label{sec1}

The main objective of this paper consists in studying the relation
between fractional smoothness of the terminal condition of a BSDE and
the $\LLL_2$-variation
of its according solution.

A motivation to investigate this relation arises from the fact that the
convergence rate of the discretization error of BSDEs with Lipschitz
generator is determined
by the convergence of the discretized terminal condition to its limit
and by the $\LLL_2$-variation properties of the exact solution
$(Y,Z)$.

In the Brownian scenario, the discretization of BSDEs has been studied
by many authors, see, for example,
\cite{BallyPages,MaZhang,Zhang,BouchardTouzi,GobetLemorWarin,BouchardElieTouzi,GGG}.
If the BSDE is given by
\[
Y_t=\xi+\int_t^T F
(s,Y_{s}, Z_s)\,\mathrm{d}s-\int_t^TZ_s\,\mathrm{d}W_s,\qquad
0\le t \le T,
\]
we define the $\LLL_p$-variation
\[
\operatorname{var}_p (\xi,F,\tau) := \sup_{1 \le i \le n} \sup
_{t_{i-1}< s \le t_i} \| Y_s - Y_{t_{i-1}} \|_p
+ \Biggl( \sum_{i=1}^{n} \int
_{t_{i-1}}^{t_i} \| Z_t -
\hat{Z}_{t_{i-1}
}\|_p^2 \,\mathrm{d}t \Biggr)^{{1}/{2}},
\]
where $\tau=(t_i)_{i=0}^n$ is a deterministic time-net
$0=t_0<\cdots<t_n=T$ and
\[
\hat{Z}_{t_{i-1} } := \frac{1}{t_{i}-t_{i-1}} \mathbb{E} \biggl[ \int
_{t_{i-1}}^{t_{i}} Z_s \,\mathrm{d}s |\mathcal{F}_{t_{i-1}} \biggr].
\]
Gobet and Makhlouf \cite{GobMak} considered $\LLL_2$-regularity of
$(Y,Z)$ for a terminal condition given by
$\xi=g(X_T)$ where $g$ does not need to be
Lipschitz and $X$ denotes the forward process. In \cite{GGG}, the
$\LLL_p$-regularity of $(Y,Z)$ for $p \ge2$ was
shown if the terminal condition depends on the forward process at
finitely many time points,
$\xi=g(X_{r_1},\ldots,X_{r_m})$, and satisfies a path-dependent Malliavin
fractional smoothness condition which is weaker than
the Lipschitz condition on $g$. Using these results and following the
ideas of \cite{BouchardTouzi}, one can show that
the convergence rate of the error between the discretizations $(Y^\tau,
Z^\tau)$ and the solution $(Y,Z)$ is of order $\frac{1}{2}$, that is
\[
\operatorname{Err}_{\tau,2}(Y,Z) := \biggl\{ \sup_{0 < t \le T } \mathbb{E}\bigl|
Y_t-Y_t^\tau\bigr|^2 + \int
_0^T \mathbb{E}\bigl|Z_t-Z_t^\tau\bigr|^2
\,\mathrm{d}t \biggr\}^{{1}/{2}} \le c |\tau |^{{1}/{2}}
\]
provided that the time grid for the discretization is chosen in an
appropriate way (like in \cite{GGG}), and the discretized terminal
condition converges in this order. Without any assumptions on the
dependence of the terminal condition $\xi$ on a forward
process $X$, Hu, Nualart and Song \cite{HuNuSong} have
shown the convergence rate $\frac{1}{2}$
supposing Malliavin differentiability properties of $\xi$ (and of the
generator).

For a BSDE driven by a Poisson random measure, Bouchard and Elie
\cite{BouchardElie} proved that the convergence rate is of order
$\frac{1}{2}$ (in the non-degenerate case) if the terminal condition
is given by $\xi=g(X_T)$ where $g$ is Lipschitz.

Here we study whether the $\LLL_2$-variation would allow to
achieve the
convergence rate $\frac{1}{2}$ with a terminal condition $\xi
=g(X_{r_1},\ldots,X_{r_m})$
and whether the Lipschitz condition on
$g$ can be weakened to a Malliavin fractional smoothness condition. The
method we use allows to answer this question in the case where
$X$ is the L\'evy process itself.

In the Brownian setting, in case of a zero generator it is stated in
\cite{GeissHujo}, relation (8), that the rate~$\frac{1}{2}$ is the
best possible
as long as $\xi$ can not be represented as a linear function of $W_T$.
Moreover, in \cite{GeissHujo}, Theorem 3.5, it is shown that for
equidistant grids there is a direct connection between the convergence
rate and the index of fractional smoothness of the terminal condition.
In \cite{GGL}, Theorems 5 and 6, the same
results are recovered for $W$ replaced by a square integrable L\'evy
process $X$, even if the L\'evy process does not have a Brownian part.

The paper is organised as follows. In Section~\ref{sec2}, we describe the
setting and recall some needed facts. In Section~\ref{sec3} we recall some basic
facts about Malliavin
calculus in the L\'evy setting. Furthermore, we state a result about
Malliavin differentiability
of the solution $(Y,Z)$ to a BSDE.
Our method to show the $\LLL_2$-regularity of solutions
to BSDEs exploits the fact that their Malliavin derivative is
piece-wise constant in time. This is ensured by selecting a terminal
condition which has this property. For this purpose, we introduce a
space of
suitable terminal conditions and investigate the chaos expansion of the
according solution in Section~\ref{sec4}. Section~\ref{sec5} contains our main
result, equivalences and implications concerning the $\LLL_2$-regularity
of $(Y,Z)$. An important fact, which will be considered in
Section~\ref{sec6}, is a \mbox{sufficient} condition for the $\LLL_2$-regularity
of the solution:
a certain Malliavin fractional smoothness of the terminal condition
(in addition to our standing assumption that the generator is Lipschitz).

\section{Setting}\label{sec2}

Let $X= (X_t )_{t\in{[0,T]}}$ be an $\LLL_2$-L\'evy-process on
a complete probability space $(\Omega,\mathcal{F},\mathbb{P})$
with L\'evy-measure $\nu$. We will denote the augmented natural
filtration of $X$ by
$\mathbb{F}:= ({\mathcal{F}_t} )_{t\in{[0,T]}}$ and let
$\LLL_2
:= \LLL_2 (\Omega,\mathcal{F}_T,\mathbb{P} )$.

The L\'evy--It\^o decomposition of an $\LLL_2$-L\'evy-process $X$
can be
written as
%
\begin{equation}
\label{levy}
X_t = \gamma t + \sigma W_t + \int
_{]0,t]\times (\mathbb{R}\setminus\{0\}
 )} x\tilde{N}(\mathrm{d}s,\mathrm{d}x),
\end{equation}
where $\sigma\geq0$, $W$ is a Brownian motion and $\tilde{N}$ is the
compensated Poisson random measure corresponding to $X$.

We define the random measure $M$ by
%
\begin{equation}
\label{measureM}
M(\mathrm{d}t,\mathrm{d}x):=\sigma \,\mathrm{d}W_t\delta_0(\mathrm{d}x)+x\tilde
N(\mathrm{d}t,\mathrm{d}x).
\end{equation}
Then $\mathbb{E}M(B)^2 = \int_B\mathbh{m}(\mathrm{d}t,\mathrm{d}x)$ for $B \in
\mathcal{B}([0,T]\times\mathbb{R})$
where
\[
\mathbh{m}(\mathrm{d}t,\mathrm{d}x) =(\lambda\otimes\mu) (\mathrm{d}t,\mathrm{d}x)
\]
with
\[
\mu(\mathrm{d}x)=\sigma^2\delta_0(\mathrm{d}x)+x^2\nu(\mathrm{d}x)
\]
and $\lambda$ denotes the Lebesgue measure.
For $\ 0\leq t\leq T$, we consider the BSDE
%
\begin{equation}
\label{beq}
Y_t=\xi+\int_t^T F
(s, Y_{s}, \bar{Z}_s )\,\mathrm{d}s-\int_{]t,T]\times\mathbb{R}}Z_{s,x}M(\mathrm{d}s,\mathrm{d}x),
\end{equation}
where
\[
\bar{Z}_s=\int_{\mathbb{R}}Z_{s,x}\kappa(\mathrm{d}x)
\]
and
$\kappa(\mathrm{d}x) := \kappa'(x) \mu(\mathrm{d}x)$ such that $\kappa'\in\LLL_2(\mathbb{R}, \mathcal{B}(\mathbb{R}),\mu)$. We will use the
following assumptions:
\begin{enumerate}[($A_F$)]
\item[($A_{\xi}$)] $\xi\in\LLL_2$,
\item[($A_F$)] $\!F\dvtx \Omega\times{[0,T]}\times\mathbb{R}^2 \to\mathbb{R}$ is
jointly measurable,
adapted to $(\mathcal{F}_t)_{t\in{[0,T]}}$,
Lipschitz-continuous in the last two variables, uniformly in $\omega$
and $t$, that is,
%
\begin{equation}
\label{Lipschitzcondition2}
\bigl|F(t,y_1,z_1)-F(t,y_2,z_2)\bigr|
\leq L_f \bigl(|y_1-y_2|+|z_1-z_2|
\bigr),
\end{equation}
and satisfies
\[
\mathbb{E}\int_0^T\bigl\llvert F(t,0,0)\bigr
\rrvert ^2\,\mathrm{d}t<\infty.
\]
\end{enumerate}

For later use, we introduce spaces of stochastic processes.

\begin{definition}\label{def21}
1. Let $S$ denote the space of all $(\mathcal
{F}_t)$-adapted and c\`adl\`ag processes $(y_t)_{0\le t\le T}$ such that
\[
\llVert y\rrVert _{S}^2:=\mathbb{E}\sup
_{0\leq t\leq T} \llvert y_{t}\rrvert ^2 <\infty.
\]
2. We define $H$ as the space of all random fields $z\dvtx \Omega
\times [0,T] \times\mathbb{R}\rightarrow\mathbb{R}$ which are
measurable with respect to
$\mathcal{P}\otimes\mathcal{B}(\mathbb{R})$ (where $\mathcal{P}$
denotes the
predictable $\sigma$-algebra on $\Omega\times[0,T]$ generated
by the left-continuous $\mathbb{F}$-adapted processes) such that
\[
\llVert z\rrVert _{H}^2:=\mathbb{E}\int
_{[0,T]\times\mathbb{R}}\llvert z_{s,x}\rrvert ^2
\mathbh{m} (\mathrm{d}s,\mathrm{d}x)<\infty.
\]
The space $S\times H$ is equipped with the norm $\|(y,z) \|_{S\times
H}:=  ( \llVert y\rrVert _{S}^2+ \llVert z\rrVert _{H}^2
)^{{1}/{2}}$.
\end{definition}

A pair $(Y,Z) \in S\times H$ which satisfies (\ref{beq}) is called a
solution to the BSDE (\ref{beq}).

\begin{thm} \label{existence}
Assume $(\xi,F)$ satisfies ($A_{\xi}$) and ($A_F$). Then the BSDE
(\ref{beq}) has a unique solution $(Y,Z) \in S\times H$.
\end{thm}

This assertion can be found in Tang and Li \cite{Tangli}, Lemma 2.4 and
in Barles, Buckdahn and Pardoux \cite{bbp}, Theorem 2.1.
We next cite the stability result of \cite{bbp} comparing the distance
between solutions to the BSDE (\ref{beq}) with different
terminal conditions and generators.

\begin{thm}[(\cite{bbp}, Proposition 2.2)] \label{continuitythm} Let
$(\xi
,F)$ and $(\xi',F')$ satisfy ($A_\xi$) and ($A_F$). For the
corresponding solutions $(Y,Z)$ and $(Y',Z')$ to (\ref{beq}), it holds
\begin{eqnarray*}
&& \bigl\|Y-Y'\bigr\|_{S}^2+ \bigl\|Z-Z'
\bigr\|_{H}^2
\\
&& \quad \le C \biggl( \bigl\|\xi-\xi'\bigr\|_{\LLL_2}^2 +\int
_0^T\bigl\|F (s, Y_s,
\bar{Z}_s )-F' (s,Y_s,
\bar{Z}_s )\bigr\|_{\LLL_2}^2 \,\mathrm{d}s \biggr),
\end{eqnarray*}
where $C= C(T,L_{F'}, \mu)$.
\end{thm}

\begin{rem}\label{rem24}
According to \cite{StrickerYor}, Proposition 3 (see also \cite{Meyer}, Proposition 3  or \cite{Ankirch}, Lemma 2.2) for any $z\in
\LLL_2
(\mathbb{P}\otimes\mathbh{m})$ there exists a process
\[
{}^pz \in\LLL_2 \bigl(\Omega\times{[0,T]}\times
\mathbb{R},\mathcal {P}\otimes\mathcal{B}(\mathbb{R}), \mathbb{P}\otimes\mathbh{m} \bigr)
\]
such that for any fixed $x\in\mathbb{R}$ the function
$({}^pz)_{\cdot,x}$ is a version of the predictable projection (in the
classical sense) of $ z_{ \cdot,x}$. In the following, we will always
use this result to
get predictable projections which are measurable w.r.t. a parameter.
\end{rem}

\section{Malliavin differentiability of $(Y,Z)$}\label{sec3}

We will use the Malliavin derivative which is defined via chaos
expansions, that is series of multiple stochastic integrals.
Following It\^o \cite{ito}, for $n\geq1$ we define elementary
functions of the type
\[
f_n\bigl((t_1,x_1),\ldots,(t_n,x_n)
\bigr)=\sum_{k=1}^m a_k \prod
_{i=1}^n\mathbh{1} _{B_i^k}(t_i,x_i),
\]
where $a_k \in\mathbb{R}$, and for all $k$ the sets $B^k_i \in
\mathcal
{B}([0,T]\times\mathbb{R}), i=1,\ldots,n$ are disjoint and satisfy
$\mathbh{m}
(B_i^k)<\infty$. The multiple stochastic integral $I_n$ of order $n$ of
the elementary function $f_n$
with respect to the random measure $M$ is defined by
\[
I_n(f_n):=\sum_{k=1}^m
a_k \prod_{i=1}^n M
\bigl(B_i^k\bigr).
\]
Since the elementary functions given above are dense in
$L_2^n:=\LLL_2 ({ ({[0,T]}\times\mathbb{R}
)}^n,\mathbh{m}^{\otimes
n} )$,
by linearity and continuity of $I_n$ its domain extends to
$\LLL_2^n$.
For $n=0$, we set $\LLL_2^n:=\mathbb{R}$ and $I_0(f_0):=f_0$
for $f_0\in
\mathbb{R}$.
The properties of $I_n$ are very similar to those in the Brownian
case, especially it holds
\[
I_n(f_n) = I_n( \tilde f_n),
\]
where $\tilde f_n$ denotes the symmetrization of $f$ with respect to
the $n$ pairs $(t_1,x_1),\ldots,(t_n,x_n)$. Moreover,
\[
\mathbb{E}I_n( f_n) I_m( g_m) =
\cases{
\ds\langle\tilde f_n, \tilde
g_n \rangle_{\mathit{L}_2^n}, & \quad $n=m$,
\vspace*{3pt}\cr
0, &\quad  $n \neq m$.}
\]
Any $G \in\LLL_2$ has a chaos expansion
\[
G=\sum_{n=0}^\infty I_n(f_n)
\]
which is unique if symmetric $f_n \in\LLL_2^n$ are used (which
we will be our standing assumption from now on), and it holds
\[
\|G\|^2:= \|G\|^2_{\LLL_2}= \sum
_{n=0}^\infty n!\llVert f_n\rrVert
_{\LLL^n_2}^2.
\]
For example, for $X_s$ from (\ref{levy}) we have
%
\begin{equation}
\label{levy-chaos}
X_s = I_0(\gamma s ) + I_1(
\mathbh{1}_{[0,s]\times\mathbb{R}}) =
\gamma s + M \bigl([0,s]\times\mathbb{R}\bigr).
\end{equation}
A straightforward generalisation of \cite{Nualart}, Lemma 1.2.5, implies
($\mathbb{E}_t$ stands for the conditional expectation
$\mathbb{E}[ \cdot|\mathcal{F}_t ]$)
\[
\mathbb{E}_tG=\sum_{n=0}^\infty
I_n(f_n\mathbh{1}_{[0,t]^n}).
\]

The space ${\mathbb{D}_{1,2}}$ consists of all random variables
$G \in \LLL_2$ such\vspace*{-2pt} that
\[
\|G\|^2_{{\mathbb{D}_{1,2}}}:= \sum_{n=0}^\infty(n+1)!
\llVert f_n\rrVert _{\mathit{L}^n_2}^2<\infty.
\]
For $G \in{\mathbb{D}_{1,2}}$ the Malliavin derivative
\[
\mathcal{D}G \in\LLL_2(\mathbb{P}\otimes\mathbh{m}):=
\LLL_2 \bigl(\Omega\times[0,T] \times\mathbb{R},\mathcal
{F}_T\otimes\mathcal{B}\bigl( [0,T]\times\mathbb{R}\bigr),\mathbb{P}
\otimes\mathbh{m} \bigr)
\]
is given\vspace*{-3pt} by
\[
\mathcal{D}_{t,x}G(\omega):=\sum_{n=1}^\infty
nI_{n-1} \bigl(f_n \bigl((t,x), \cdot \bigr) \bigr) (
\omega),
\]
for $\mathbb{P}\otimes\mathbh{m}$-a.e. $(\omega,t,x)\in\Omega
\times
{[0,T]}\times\mathbb{R}$.
For example, for $X_s$ from (\ref{levy}) representation (\ref{levy-chaos}) implies
%
\begin{equation}
\label{levy-derivative}
\mathcal{D}_{t,x} X_s =
\mathcal{D}_{t,x} I_1 (\mathbh {1}_{[0,s]\times\mathbb{R}}) =
\mathbh{1}_{[t,T]}(s)\qquad \mbox{for } \mathbb{P}\otimes\mathbh{m}\mbox{-a.e. }(\omega ,t,x)\in\Omega \times{[0,T]}\times\mathbb{R}.
\end{equation}

For random variables in ${\mathbb{D}_{1,2}}$ there is an explicit
expression for the
integrand in the formulation of the predictable representation property
(for an introduction to stochastic integration w.r.t. random measures
see, e.g., \cite{Applebaum}).

\begin{lem}[(Clark--Ocone--Haussmann formula \cite{suv}, Theorem 10)]
\label{lem31}
Assume\hspace*{-2pt} $G\in{\mathbb{D}_{1,2}}$. Then
%
\begin{equation}
\label{COH}
G=\mathbb{E}G + \int_{[0,T]\times\mathbb{R}} {}^p(
\mathcal{D}G) _{t,x}M(\mathrm{d}t,\mathrm{d}x).
\end{equation}
\end{lem}

The Malliavin derivative $\mathcal{D}_{\cdot,0}$ can be interpreted
as a
Malliavin derivative in the Brownian setting with values in the
$\LLL_2$-space of random variables
depending on the jump part of the L\'evy process (see \cite{alos,suv2}). On the other hand, for $x\neq0$, the Malliavin derivative
$\mathcal{D}_{\cdot,x}$ behaves like a difference quotient
(see \cite{alos,suv2}).
This is also illustrated by the next lemma which contains formulae for
the Malliavin derivative of differentiable Lipschitz functions
depending on random variables in $\mathbb{D}_{1,2}$.

\begin{lem} \label{composition}
Let $f \in C^1(\mathbb{R}^n;\mathbb{R})$ with bounded partial
derivatives. If
$G_1,\ldots,G_n\in{\mathbb{D}_{1,2}}$ then $f(G_1,\ldots,G_n)\in
{\mathbb{D}_{1,2}}$
and
\begin{enumerate}[(ii)]
\item[(i)] for\vspace*{-3pt} $x=0$
it holds
\[
\mathcal{D}_{t,0}f(G_1,\ldots,G_n)=\sum
_{i=1}^n (\partial_i f)
(G_1,\ldots,G_n) \mathcal{D} _{t,0}G_i,
\]
for $\mathbb{P}\otimes\lambda$-a.e. $(\omega,t)$;
\item[(ii)] for\vspace*{-2pt} $x\neq0$ we get the difference quotient
\[
\mathcal{D}_{t,x}f(G_1,\ldots,G_n)=
\frac{f(G_1+x\mathcal{D}_{t,x}G_1,\ldots,G_n+x\mathcal{D}
_{t,x}G_n)-f(G_1,\ldots,G_n)}{x},
\]
for $\mathbb{P}\otimes\mathbh{m}$-a.e. $(\omega,t,x)$.
\end{enumerate}
\end{lem}

\begin{pf}
Assertion (i) follows immediately from \cite{suv2}, Proposition~3.5,
combined with \cite{suv2}, Proposition~3.3  and
\cite{Nualart}, Proposition~1.2.3. Assertion (ii) is a straightforward
extension of \cite{GL}, Lemma~5.1.
\end{pf}

We will make use of the following properties for the Malliavin
derivative \cite{ImkellerDelong}, Lemmas \mbox{3.1--3.3}.

\begin{lem} \label{ImkD}
\textup{(i)} Let $G\in{\mathbb{D}_{1,2}}$. Then for $0\leq s\leq T$,
$\mathbb{E}_s G\in{\mathbb{D}_{1,2}}$ and
\[
\mathcal{D}_{t,x}\mathbb{E}_s G=\mathbb{E}_s(
\mathcal{D}_{t,x}G) \mathbh{1}_{[0,s]}(t),\qquad \mathbb{P}\otimes
\mathbh{m}\mbox{-a.e.}\\[-18pt]
\]
\begin{longlist}[(iii)]
\item[(ii)]  Let $F\dvtx \Omega\times{[0,T]}\times\mathbb{R}\to
\mathbb{R}$ be a product
measurable and adapted process, $\rho$ a finite measure on
$([0,T]\times
\mathbb{R}, \mathcal{B}([0,T]\times\mathbb{R}))$ such that the
conditions
\begin{enumerate}[(iii) (a)]
\item[(a)]  \vspace*{2pt}$\mathbb{E}\int_{[0,T]\times\mathbb{R}}|F(s,y)|^2\rho
(\mathrm{d}s,\mathrm{d}y )<\infty$,

\item[(b)] \vspace*{2pt}$F(s,y)\in{\mathbb{D}_{1,2}}$ for $\rho$-a.e. $(s,y)\in
{[0,T]}\times\mathbb{R}$,

\item[(c)] $\mathbb{E}\int_{[0,T]\times\mathbb{R}}\int_{[0,T]\times
\mathbb{R}} |\mathcal{D}_{t,x}F(s,y)|^2\rho
(\mathrm{d}s,\mathrm{d}y )\mathbh{m}(\mathrm{d}t,\mathrm{d}x)<\infty$
\end{enumerate}
are satisfied. Then
\[
\int_{[0,T]\times\mathbb{R}}F(s,y)\rho(\mathrm{d}s,\mathrm{d}y)\in{\mathbb{D}_{1,2}},
\]
and the differentiation rule
\[
\mathcal{D}_{t,x}\int_{[0,T]\times\mathbb{R}}F(s,y)\rho (\mathrm{d}s,\mathrm{d}y)=\int
_{[0,T]\times\mathbb{R}}\mathcal{D} _{t,x}F(s,y)\rho(\mathrm{d}s,\mathrm{d}y)
\]
holds for $\mathbb{P}\otimes\mathbh{m}$-a.e. $(\omega,t,x)\in
\Omega\times
{[0,T]}\times\mathbb{R}$.

\item[(iii)] Let $F\dvtx\Omega\times{[0,T]}\times\mathbb{R}\to
\mathbb{R}$ be a
predictable process satisfying
\[
\mathbb{E}\int_{[0,T]\times\mathbb{R}}\bigl|F(s,y)\bigr|^2\mathbh{m}(\mathrm{d}s,\mathrm{d}y )<\infty.
\]
Then conditions \textup{(a)}--\textup{(c)} of \textup{(ii)} are satisfied for $\rho=\mathbh{m}$
if and only if
\[
\int_{[0,T]\times\mathbb{R}}F(s,y)M(\mathrm{d}s,\mathrm{d}y )\in{\mathbb{D}_{1,2}}.
\]
In this case, the formula
\[
\mathcal{D}_{t,x}\int_{[0,T]\times\mathbb
{R}}F(s,y)M(\mathrm{d}s,\mathrm{d}y )=F(t,x)+
\int_{[0,T]\times\mathbb{R}}\mathcal{D} _{t,x}F(s,y)M(\mathrm{d}s,\mathrm{d}y )
\]
holds $\mathbb{P}\otimes\mathbh{m}$-a.e.
\end{longlist}
\end{lem}

The following theorem is concerned with Malliavin differentiability of
the solution to a BSDE of the form
%
\begin{equation}
\label{beq2}
Y_t = \xi+\int_t^T
f (s, X_s, Y_{s}, \bar{Z}_s )\,\mathrm{d}s-\int
_{{]t,T]}\times\mathbb{R}}Z_{s,x}M(\mathrm{d}s,\mathrm{d}x),\qquad 0\le t\le T,
\end{equation}
where we will assume
\begin{enumerate}[($A_f 1$)]
\item[($A_f$)]
$f \in\mathcal{C}([0,T]\times\mathbb{R}^3)$ is Lipschitz-continuous
in $(x,y,z)$,
uniformly in $t$, that is,
%
\begin{equation}\label{Lipschitzcondition}
\bigl|f(t,x_1,y_1,z_1)-f(t,x_2,y_2,z_2)\bigr|
\le L_f \bigl(|x_1- x_2|+|y_1-y_2|+|z_1-z_2|\bigr).
\end{equation}
\item[($A_f 1$)] $f$ satisfies ($A_f $) and
$f \in\mathcal{C}^{0,1,1,1}([0,T]\times\mathbb{R}^3)$.
\end{enumerate}

Note that (\ref{beq2}) is a special form of (\ref{beq}), and
$F(\omega,t,y,z) :=f(t,X_t(\omega),y,z)$ satisfies
($A_F $) if $f$ does satisfy ($A_f$).

\begin{thm}\label{diffthm}
Let $\xi\in\mathbb{D}_{1,2}$ and assume ($A_f 1$). Then the following
assertions hold.
\begin{longlist}[(iii)]
\item[(i)]\label{alef} For $\mathbh{m}$-a.e. $(r,v) \in[0,T]\times
\mathbb{R}$, there
exists a unique solution $(U^{r,v},V^{r,v})$ $\in S\times H $ to the BSDE
%
\begin{eqnarray}
U^{r,v}_t &=& \mathcal{D}_{r,v}
\xi+ \int_t^T F_{r,v} \bigl( s,
U^{r,v}_s, \bar {V}^{r,v} _s \bigr)
\,\mathrm{d}s 
-\int_{{]t,T]}\times\mathbb{R}} V^{r,v}_{s,x}M(\mathrm{d}s,\mathrm{d}x),
\qquad t \in[r,T],
\nonumber
\\[-8pt]
\label{U-V-equation}\\[-8pt]
\nonumber
U^{r,v}_t &=& V^{r,v}_{s,x} =0,\qquad t
\in[0,r),
\end{eqnarray}
where $ \bar{V}^{r,v} _s := \int_\mathbb{R}V^{r,v} _{s,x}\kappa
(\mathrm{d}x)$,
\[
F_{r,0} \bigl(s, U^{r,0}_s, \bar{V}^{r,0}
_s \bigr) := \bigl\langle\nabla f (s,X_s,Y_s,
\bar{Z}_s ), \bigl(\mathbh{1} _{[r,T]}(s),
U^{r,0}_s, \bar{V}^{r,0} _s \bigr)
\bigr\rangle,
\]
with $\nabla= (\partial_x, \partial_y,\partial_z)$, and
\begin{eqnarray}
F_{r,v} \bigl(s, U^{r,v}_s, \bar{V}^{r,v}
_s \bigr)& :=& \frac{1}{v}\bigl[f \bigl(s,X_s+v
\mathbh {1}_{[r,T]}(s),Y_s+vU^{r,v}_s,
\bar {Z}_s +v \bar{V}^{r,v}_s
\bigr)
-f (s,X_s,Y_s, \bar{Z}_s )\bigr] \nonumber\\
\eqntext{\mbox{for } v\neq0.}
\end{eqnarray}

\item[(ii)]\label{bet} For the solution $(Y,Z)$ of (\ref{beq2}) it holds
%
\begin{equation}
\label{Y-and-Z-in-D12}
Y \in\LLL_2\bigl([0,T];{\mathbb{D}_{1,2}}
\bigr), \qquad  Z \in\mathit{L}_2\bigl([0,T]\times\mathbb{R};{
\mathbb{D}_{1,2}}\bigr),
\end{equation}
and $\mathcal{D}_{r,v}Y$ admits a c\`adl\`ag version for $\mathbh{m}$-a.e. $(r,v) \in
[0,T]\times\mathbb{R}$.
\item[(iii)]\label{gimmel} $(\mathcal{D}Y,\mathcal{D}Z)$ is a version of
$(U,V)$, that is,
for $\mathbh{m}$-a.e. $(r,v)$ it solves
%
\begin{eqnarray}
\mathcal{D}_{r,v}Y_t &= &
\mathcal{D}_{r,v}\xi+\int_t^T
F_{r,v} \biggl(s,\mathcal{D}_{r,v}Y_s, \int
_\mathbb{R}\mathcal{D}_{r,v}Z_{s,x}\kappa(\mathrm{d}x)
\biggr)\,\mathrm{d}s
\nonumber
\\[-8pt]
\label{diffrep}
\\[-8pt]
\nonumber
&&{}-\int_{{]t,T]}\times\mathbb{R}} \mathcal{D}_{r,v}
Z_{s,x}M(\mathrm{d}s,\mathrm{d}x),\qquad t \in[r,T].
\end{eqnarray}
\item[(iv)]\label{dalet} The process ${}^p (
(\mathcal{D}
_{r,v}Y_r )_{r\in{[0,T]}, v\in\mathbb{R}} )$ is a
version of $Z$
where we set
\[
\mathcal{D}_{r,v}Y_r(\omega):=\lim_{t\searrow r}
\mathcal {D}_{r,v}Y_t(\omega)
\]
for all $(r,v,\omega)$ for which $ \mathcal{D}_{r,v}Y$ is c\`adl\`ag
and $\mathcal{D}
_{r,v}Y_r(\omega):=0$ otherwise.
\end{longlist}
\end{thm}

In the setting of time, delayed BSDEs a similar result was proved by
Delong and Imkeller \cite{ImkellerDelong} assuming that the time
horizon $T$
or the Lipschitz constant $L_f$ of the generator are sufficiently small.
For the convenience of the reader a proof of Theorem \ref{diffthm} is
contained in the preprint version \cite{GeissSteinicke}.

\section{Chaotic representation of $(Y,Z)$}\label{sec4}

The goal of this section is to investigate properties of the chaos
expansions of the solution $(Y,Z)$ to (\ref{beq2})
with terminal values $\xi$ of the following form:

Let $r_0=0<r_1<\cdots<r_m=T$ be a partition of $[0,T]$. Define
$\Lambda
_k:={]r_{k-1},r_k]}$ for $k=1, \ldots,m$ and $V_m^n := \{1,\ldots,m\}^n$.
The set of cuboids
\[
\bigl\{ \Lambda_\alpha:=\Lambda_{\alpha_1}\times\cdots\times
\Lambda _{\alpha_n}\dvt  \alpha =(\alpha_1, \ldots,
\alpha_n) \in V_m^n \bigr\}
\]
forms a partition of $]0,T]^n$.
Furthermore, we let
\begin{eqnarray*}
\mathbb{H} & := &\Biggl\{\xi=\sum_{n=0}^\infty
I_n(f_n)\in\LLL_2\dvt  \exists
g_n^{\alpha} \in \LLL_2\bigl(
\mathbb{R}^n, \mu^{\otimes n}\bigr) \mbox{ such that }
\\
&&\hspace*{7pt}f_n\bigl((t_1,x_1),\ldots,(t_n,x_n)
\bigr)=\sum_{\alpha\in
V_m^n}g_n^{\alpha}(x_1,\ldots,x_n) \mathbh{1}_{\Lambda_\alpha}(t_1,\ldots,t_n)\Biggr\}.
\end{eqnarray*}
Hence, on each cuboid $\Lambda_\alpha$ the function $f_n$ is constant in
$(t_1,\ldots, t_n)$.
In particular, this space contains random variables of the form
\[
g(X_{r_m}-X_{r_{m-1}},\ldots,X_{r_1}-X_0)
\in\LLL_2,
\]
where $g$ is a Borel function (see \cite{baumgartner:geiss:2011}).

The benefit to consider terminal conditions from $\mathbb{H}$ lies in
the fact that $t \mapsto\mathcal{D}_{tx}\xi$ is a.e. constant as
long as $t$ is
within an interval $\Lambda_k$. This property will be used several times
below, especially in the proofs of Lemmas~\ref{Hsmoothlemma}--\ref{savinglemma}.

\begin{rem}\label{approximationthm}
By convolution with mollifiers, we construct for any function $f\in
\mathcal{C}({[ 0,T]} \times\mathbb{R}^3)$ satisfying ($A_f$) a sequence
$(f_n)_{n\geq0}$
converging uniformly to $f$ in $\mathcal{C}({[0,T]}\times\mathbb{R}^3)$, such
that for all
$n \in\mathbb{N}$ and $t\in{[0,T]}$ we have $f_n(t,\cdot)\in
\mathcal{C}^\infty(\mathbb{R}^3)$, and $f_n$ satisfies the Lipschitz-condition
(\ref{Lipschitzcondition}) with the same constant
$L_f$ for all $n$.

Let $(\xi_n)_{n\geq0} \subseteq{\mathbb{D}_{1,2}}$ be\vspace*{1pt} a sequence
converging to $\xi$
in $\LLL_2$. By $(Y^n,Z^n)$, we will denote the solution to
(\ref{beq2})
with terminal condition $\xi_n$ and generator $f_n$. Then Theorem~\ref{continuitythm} implies that
%
\begin{equation}
\label{Yn-Znconvergence}
\bigl(Y^n,Z^n\bigr)\to(Y,Z) \qquad\mbox{if } n\to
\infty\mbox{ in } S\times H.
\end{equation}
\end{rem}

If $\xi\in\mathbb{H}$,  then the solution $(Y,Z)$ has a chaos expansion
which resembles those of the elements of $\mathbb{H}$.

\begin{thm}\label{Hthm}
Let ($A_f$) hold.
For $\xi\in\mathbb{H}$ the chaos expansion of $(Y,Z) \in S\times H$ has
the representation
%
\begin{equation}
\label{y-rep}
Y_t = \sum_{n=0}^\infty
I_n \biggl(\sum_{\alpha\in V_m^n} \varphi
_n^{\alpha}(t) \mathbh{1}_{\Lambda_\alpha\cap]0,t]^n} \biggr),\qquad \mathbb{P}
\otimes \lambda\mbox{-a.e.},
\end{equation}
where $\varphi_n^{\alpha} \dvtx [0,T]\times\mathbb{R}^n \to\mathbb{R}$
is jointly
measurable, $\varphi_n^{\alpha}(t) \in\LLL_2(\mathbb{R}^n,
\mu^{\otimes n})$
and
%
\begin{equation}
\label{z-rep}
Z_{t,x} =\sum_{n=0}^\infty
I_n \biggl( \sum_{\alpha\in V_m^n} \psi
_n^{\alpha}(t,x) \mathbh{1}_{\Lambda_\alpha\cap]0,t]^n} \biggr),\qquad
\mathbb{P}\otimes\mathbh{m} \mbox{-a.e.},
\end{equation}
where $\psi_n^{\alpha} \dvtx  [0,T]\times\mathbb{R}^{n+1} \to\mathbb{R}$
is jointly
measurable and $\psi_n^{\alpha}(t) \in\LLL_2(\mathbb{R}^{n+1}, \mu^{\otimes
n+1})$.
\end{thm}

\begin{pf}
We may use Remark \ref{approximationthm} and approximate $\xi\in
\mathbb{H}$ by a sequence $(\xi_n)_n \subseteq\mathbb{H} \cap
{\mathbb{D}_{1,2}}$
and $f$ by $(f_n)_n $ satisfying $(A_f 1)$. Since the convergence in
$S\times H$ implies convergence w.r.t. the norm
%
\begin{equation}
\label{normprod}
\bigl|\hspace{-2pt} \bigr\|(y,z)\bigr|\hspace{-2pt}\bigr\|:= \bigl( \| y \|
^2_{\LLL_2
(\mathbb{P}\otimes\lambda)} + \|z\|^2_{H}
\bigr)^{{1}/{2}},
\end{equation}
and the space of processes $(y,z)$ with representations (\ref{y-rep})
and (\ref{z-rep}) is closed with respect
to the norm (\ref{normprod}) we only need to show that the assertion
holds for any solution $(Y^n,Z^n)$
w.r.t. $(\xi_n, f_n)$.
Hence we may assume that $\xi\in\mathbb{H} \cap{\mathbb{D}_{1,2}}$
and $f\in\mathcal{C}
^{0,1,1,1}([0,T]\times\mathbb{R}^3)$ such that $\partial_x f
,\partial_y f$
and $\partial_z f$ are bounded by $L_f$.
According to Theorem \ref{diffthm}, we can differentiate (\ref{beq2})
and obtain for $\mathbh{m}$-a.e. $(t,x)$ and all $s \in[t,T]$ that
\[
\mathcal{D}_{t,x}Y_s = \mathcal{D}_{t,x}\xi+
\int_s^T \mathcal{D}_{t,x}
f(r,X_r,Y_r,\bar{Z}_r)\,\mathrm{d}r-\int
_{{]s,T]}\times\mathbb{R}}\mathcal{D}_{t,x}Z_{r,y}M(\mathrm{d}r,\mathrm{d}y).
\]

Theorem \ref{diffthm} yields that $Z$ is a version of ${}^p (\mathcal{D}_{t,x}Y_t )$, hence
\[
Z_{t,x}= \mathcal{D}_{t,x}Y_t, \qquad \mathbb{P}\otimes
\mathbh{m}\mbox{-a.e.}
\]

We define the recursion
%
\begin{eqnarray}
Y^0_s &:=& 0, \qquad Z^0_{s,y}:=0,
\nonumber
\\[-8pt]
\label{rek2y}
\\[-8pt]
\nonumber
\mathcal{Y}^{k+1}_{s} &:=& \xi+\int
_s^Tf \bigl(r,X_{r},Y^k_{r},
\bar {Z}^k_r \bigr)\,\mathrm{d}r, \qquad Y^{k+1}:=
{}^o \bigl( \mathcal{Y}^{k+1}\bigr),
\end{eqnarray}
where ${}^o$ denotes the optional projection, which is
according to \cite{DellacherieMeyer}, Theorem 47 and Remark 50,  c\`adl\`ag.
Since $Y_u^{k+1} = \mathbb{E}_u \mathcal{Y}^{k+1}_u$ $\mathbb{P}$-a.s.
one gets by Lemma \ref{ImkD}
%
\begin{eqnarray}
\mathcal{D}_{s,y}Y_u^{k+1} &=&
\mathcal{D}_{s,y} \mathbb{E}_u\xi+ \mathcal{D}_{s,y}
\mathbb {E}_u \int_u^T f
\bigl(r,X_r,Y^k_r,\bar {Z}^k_r
\bigr)\,\mathrm{d}r
\nonumber
\\[-8pt]
\label{D-of-Yk}\\[-8pt]
\nonumber
&=& \mathbb{E}_u \mathcal{D}_{s,y} \xi+
\mathbb{E}_u \int_u^T
\mathcal{D}_{s,y} f\bigl(r,X_r,Y^k_r,
\bar {Z}^k_r\bigr)\,\mathrm{d}r,\qquad u \in[s,T].
\end{eqnarray}
Since $ \mathcal{D}_{s,y} \xi+ \int_u^T \mathcal{D}_{s,y}
f(r,X_r,Y^k_r,\bar{Z}^k_r)\,\mathrm{d}r,
u \in[s,T]$, has continuous paths for a.e. $(s,y)$ we can again apply
\cite{DellacherieMeyer}, Theorem 47 and Remark 50,  to get a c\`adl\`ag
optional projection. Hence, we may define the set
\[
A_k:= \bigl\{ (s,y) \in[0,T]\times\mathbb{R}\dvt \mbox{the RHS of (\ref{D-of-Yk}) is c\`adl\`ag on } [s,T] \ \mathbb{P}\mbox{-a.s.} \bigr\}
\]
and assume a pathwise c\`adl\`ag version of $\mathcal{D}_{s,y}
Y^{k+1}$ for any
$(s,y) \in A_k$ while we let $\mathcal{D}_{s,y} Y^{k+1}$ be zero otherwise.
In this sense, we can set
%
\begin{equation}
\label{rek2z}
\mathcal{Z}^{k+1}_{s,y}:= \lim
_{t \searrow s} \mathcal{D}_{s,y} Y^{k+1}_t, \qquad Z^{k+1}:= {}^p \bigl(\mathcal{Z}^{k+1}
\bigr)
\end{equation}
for $k=0,1,2,\ldots\,$.

The process $Y^{k+1}$ has a c\`adl\`ag version, therefore, $(Y^k,Z^k)
\in S\times H$ for all $k \in\mathbb{N}$.
In the proof of \cite{Tangli}, Theorem \ref{existence}, it is shown
that $(Y^k,Z^k)$ converges to $(Y,Z)$
with respect to the norm (\ref{normprod}).

Consequently, we only need to show that (\ref{y-rep}) and (\ref{z-rep})
hold for $(Y^k,Z^k)$.

For fixed $t\in \,]0,T[$, we describe (\ref{y-rep}) by introducing the space
\[
\mathbb{H}_t := \Biggl\{ \sum_{n=0}^\infty
I_n \biggl(\sum_{\alpha\in V_m^n}
g_n^{\alpha} \mathbh{1}_{\Lambda_\alpha\cap]0,t]^n} \biggr) \in
\LLL_2 \dvt  g_n^{\alpha} \in\LLL_2
\bigl(\mathbb{R}^n, \mu^{\otimes n}\bigr) \Biggr\}.
\]
From \cite{baumgartner:geiss:2011}, one can conclude the following fact.

\begin{lem} \label{Baumgartner-Geiss}
For any Borel function $h \dvtx \mathbb{R}^d \to\mathbb{R}$ and $\xi_1,
\ldots,\xi_d \in
\mathbb{H}_t$
such that $h( \xi_1, \ldots,\break \xi_d) \in\LLL_2$ it holds
$h(\xi_1, \ldots, \xi_d) \in\mathbb{H}_t$.
\end{lem}

Assume now that (\ref{y-rep}) and (\ref{z-rep}) hold for $Y^k$ and
$Z^k$, respectively. We have
%
\begin{eqnarray}
\bar{Z}^k_t &=&\int_{\mathbb{R}}\sum
_{n=0}^\infty I_n \biggl(\sum
_{\alpha\in
V_m^n}\psi_n^{\alpha}(t,x)
\mathbh{1}_{\Lambda_\alpha\cap
[0,t]^{\otimes n}} \biggr)\kappa(\mathrm{d}x)
\nonumber\\
&=& \sum_{n=0}^\infty I_n \biggl(
\sum_{\alpha\in V_m^n}\int_{\mathbb{R}} \psi
_n^{\alpha}(t,x) \kappa(\mathrm{d}x) \mathbh{1}_{\Lambda_\alpha\cap
[0,t]^{\otimes
n}} \biggr)
\\
&=&\sum_{n=0}^\infty I_n \biggl(
\sum_{\alpha\in V_m^n }{\bar{ \psi }}_n^{\alpha}(t)
\mathbh{1}_{\Lambda_\alpha\cap[0,t]^{\otimes n}} \biggr) \in \mathbb{H}_t.\nonumber
\end{eqnarray}

From Lemma \ref{Baumgartner-Geiss}, it follows that $
f(t,X_{t},Y^k_{t},\bar{Z}^k_t) \in\mathbb{H}_t$ that is,
\[
f \bigl(t,X_{t},Y^k_{t},
\bar{Z}^k_t \bigr)=\sum_{n=0}^\infty
I_n \biggl(\sum_{\alpha\in V_m^n}g_n^{\alpha}(t)
\mathbh{1}_{\Lambda_\alpha
\cap
]0,t]^{\otimes n} } \biggr),
\]
with $g_n^{\alpha}(t) \in\LLL_2 (\mathbb{R}^n, \mu
^{\otimes n})$.
Because $ f(\cdot,X_\cdot,Y^k_\cdot,\bar{Z}^k_\cdot) $
is square integrable w.r.t. $\mathbb{P}\otimes\lambda$ on $\Omega
\times
[0,T]$ one can show that the $g_n^{\alpha}$ can be chosen jointly
measurable. This implies
\begin{eqnarray*}
\mathbb{E}_t \int_t^Tf
\bigl(r,X_r,Y^k_r,\bar{Z}^k_r
\bigr)\,\mathrm{d}r &=& \int_t^T \sum
_{n=0}^\infty I_n \biggl(\sum
_{\alpha\in
V_m^n}g_n^{\alpha}(r)\mathbh{1}_{\Lambda_\alpha\cap]0,t]^{\otimes
n} }
\biggr) \,\mathrm{d}r
\\
&=& \sum_{n=0}^\infty I_n
\biggl(\sum_{\alpha\in V_m^n} \int_t^T
g_n^{\alpha}(r) \,\mathrm{d}r \mathbh{1}_{\Lambda_\alpha\cap]0,t]^{\otimes n}
} \biggr).
\end{eqnarray*}

From (\ref{rek2y}), we have that $Y_t^{k+1} = \mathbb{E}_t \mathcal
{Y}^{k+1}_t$ $\mathbb{P}\mbox{-a.s.}$ and since $\mathbb{E}_t \xi\in\mathbb
{H}_t$ we
conclude representation~(\ref{y-rep}) for $Y_t^{k+1}$.
To find out the representation of $Z^{k+1}$, we will use (\ref
{rek2z}). Let $\underline\alpha:=(\alpha_2, \ldots, \alpha_n)$.
Assuming\vspace*{1pt}
$\xi= \sum_{n=0}^\infty I_n (\sum_{\alpha\in V_m^n} \hat
{g}_n^{\alpha}\mathbh{1}_{\Lambda_\alpha} )$ with symmetric
$f_n=\sum_{\alpha\in V_m^n} \hat{g}_n^{\alpha}\mathbh{1}_{\Lambda_\alpha}
$ we get by
Lemma \ref{ImkD}
for $\mathbb{P}\otimes\mathbh{m}$-a.e. $(t,y,\omega) \in \,]0,s]\times\mathbb{R}\times\Omega
$ that
\begin{eqnarray*}
\mathcal{D}_{t,y} Y^{k+1}_s &=&
\mathcal{D}_{t,y} \mathbb{E}_s \xi+ \mathcal{D}_{t,y}
\mathbb {E}_s \int_s^T f
\bigl(r,X_{r},Y^k_{r},\bar{Z}^k_r
\bigr)\,\mathrm{d}r
\\
&=&\mathcal{D}_{t,y} \mathbb{E}_s \xi+
\mathcal{D}_{t,y} \int_s^T \sum
_{n=0}^\infty I_n \biggl(\sum
_{\alpha\in V_m^n}g_n^{\alpha}(r)
\mathbh{1}_{\Lambda_\alpha\cap
]0,s]^{\otimes n}
} \biggr) \,\mathrm{d}r
\\
&=& \sum_{n=1}^\infty n I_{n-1}
\biggl(\sum_{\alpha\in V_m^n} \hat {g}_n^{\alpha}(y,
\cdot) \mathbh{1}_{\Lambda_{\alpha_1}}(t) \mathbh {1}_{\Lambda_{\underline\alpha} \cap
]0,s]^{\otimes(n-1)} } \biggr)
\\
&& {}+ \int_s^T \sum
_{n=1}^\infty n I_{n-1} \biggl(\sum
_{\alpha\in
V_m^n}g_n^{\alpha}(r,y,\cdot)
\mathbh{1}_{\Lambda_{\alpha_1}}(t) \mathbh{1}_{\Lambda_{\underline\alpha} \cap
]0,s]^{\otimes(n-1)} } \biggr) \,\mathrm{d}r,
\end{eqnarray*}
where we again have chosen symmetric integrands $\sum_{\alpha\in V_m^n}
g_n^{\alpha}(r)\mathbh{1}_{\Lambda_\alpha\cap]0,s]^{\otimes n} }$.
One easily checks the $\LLL_2$-convergence
\begin{eqnarray*}
\lim_{s \searrow t}\mathcal{D}_{t,y}Y^{k+1}_s
& = & \sum_{n=1}^\infty n I_{n-1}
\biggl(\sum_{\alpha\in V_m^n} \hat {g}_n^{\alpha}(y,
\cdot)\mathbh{1}_{\Lambda_{\alpha_1}}(t) \mathbh {1}_{\Lambda_{\underline\alpha} \cap
]0,t]^{\otimes(n-1)} } \biggr)
\\
&& {}+ \int_t^T \sum
_{n=1}^\infty n I_{n-1} \biggl(\sum
_{\alpha\in
V_m^n}g_n^{\alpha}(r,y,\cdot)
\mathbh{1}_{\Lambda_{\alpha_1}}(t) \mathbh{1}_{\Lambda_{\underline\alpha} \cap
]0,t]^{\otimes(n-1)} } \biggr) \,\mathrm{d}r.
\end{eqnarray*}

If we consider the c\`adl\`ag version of $\mathcal{D}_{t,y} Y^{k+1}$,
we obtain
the same expression for the pathwise limit, that is, $\mathbb{P}$-a.s.
\begin{eqnarray*}
Z^{k+1}_{t,y} & = & \lim_{s\searrow t}
\mathcal{D}_{t,y}Y^{k+1}_s
\\
& = & \sum_{n=1}^\infty n I_{n-1}
\biggl( \sum_{\alpha\in V_m^n} \biggl[\hat{g}_n^{\alpha}(y,
\cdot)+\int_t^T g_n^{\alpha}(r,y,
\cdot) \,\mathrm{d}r \biggr] \mathbh{1}_{\Lambda_{\alpha_1}}(t) \mathbh{1}_{\Lambda_{\underline
\alpha} \cap]0,t]^{\otimes(n-1)}} \biggr).
\end{eqnarray*}
\upqed\end{pf}

\section{$L_2$-variation of $(Y,Z)$}\label{sec5}

The next theorem is our main statement, which allows conclusions on the
$\LLL_2$-regularity of the solutions to BSDE (\ref{beq2}) by observing
regularity properties of
$Y_{r_k}$ for fixed time points $r_0=0<r_1<\cdots<r_m=T$.

\begin{thm}\label{mainthm}
Assume that $(A_f)$ is satisfied and $\xi\in\mathbb{H}$. Let $k\in\{
1,\ldots,m\}$ and $ \theta_k \in \,]0,1]$.
For the solution $(Y,Z)$ of (\ref{beq2}), consider the following assertions:
\begin{enumerate}[(iii)]

\item[(i)]\label{exp-of-Y} There is some $c_1>0$ such that for all $s\in
{[r_{k-1},r_k]}$,
\[
\llVert Y_{r_k}-\mathbb{E}_{s}Y_{r_k}\rrVert
^2\leq c_1(r_k-s)^{\theta_k}.
\]
\item[(ii)]\label{difference-of-Y} There is some $c_2>0$ such that for all
$r_{k-1} \le s<t \le r_k$,
\[
\llVert Y_{t}-Y_{s}\rrVert ^2\leq
c_2\int_s^t(r_k-r)^{\theta_k-1}\,\mathrm{d}r.
\]
\item[(iii)]\label{Z}There is some $c_3>0$ such that for $\lambda$-a.e.
$s\in
{[r_{k-1},r_k[}$,
\[
\llVert Z_{s,\cdot}\rrVert _{\LLL_2(\mathbb{P}\otimes\mu
)}^2\leq
c_3(r_k-s)^{\theta_k-1}.
\]
\item[(iv)]\label{difference-of-Z} There is some $c_4>0$ and a Borel set
$N_k$ with $\lambda(N_k)=0$ such that for all $s, t\in[r_{k-1},r_k[
\setminus N_k$ with $s<t$ and for all $h \in\LLL_2(\mu)$ it holds
\[
\biggl\llVert \int_\mathbb{R} (Z_{t,x}-Z_{s,x}
)h(x)\mu (\mathrm{d}x)\biggr\rrVert ^2\le\| h\|^2_{\LLL_2(\mu)}
c_4\int_s^t(r_k-r)^{\theta_k-2}\,\mathrm{d}r.
\]
\end{enumerate}
Then it holds that
\[
\mathrm{(i)} \Leftrightarrow\mathrm{(ii)}\Leftrightarrow \mathrm{(iii)}
\Rightarrow \mathrm{(iv)}.
\]
\end{thm}

\begin{rem}
\label{frac-smooth-comment}
(i) Analogously to \cite{GGG}, Definition 1, we may introduce the
concept of \textit{path-dependent fractional smoothness}:
fix $\Theta=(\theta_1,\ldots,\theta_m) \in\,]0,1[^m$.
If (Y, Z) is the solution to BSDE~(\ref{beq2}) with generator f and
terminal condition $\xi\in\mathbb{H}$, we
let
\[
(\xi,f) \in B^\Theta_{2,\infty}(X)
\]
provided that there is some $c > 0$ such that
\[
\llVert Y_{r_k}-\mathbb{E}_{s}Y_{r_k}\rrVert
^2\leq c (r_k-s)^{\theta
_k},\qquad  r_{
k-1} \le s<
r_k, k=1,\ldots,m.
\]
If $ f=0$ we simply write $\xi\in B^\Theta_{2,\infty}(X)$. If,
moreover, $T=1$ and $m=1$ then the space
$B^\Theta_{2,\infty}(X)$ can be understood as the
real interpolation space $(\LLL_2,{\mathbb{D}_{1,2}})_{\theta
_1,\infty}$ which
describes {\it fractional smoothness}. For $\xi= \sum_{n=0}^\infty
I_n(f_n)$ $\in\mathbb{H}$ set $T_\xi(t) := \| \mathbb{E}_t \xi\|^2
= \sum_{n=0}^\infty\| I_n(f_n)\|^2 t^n$, and using the ideas of \cite{GeissHujo-preprint}, Proposition 3.2 and
\cite{GeissHujo}, formula (13), one can conclude that
\begin{eqnarray*}
\|\xi\|_{(\LLL_2,{\mathbb{D}_{1,2}})_{\theta_1,\infty}} &\sim_c& \|\xi\| + \sup
_{0 \le
t <1} (1-t)^{{-\theta_1}/{2}} \sqrt{T_\xi(1) -
T_\xi(t ) }
\\
&=& \|\xi\| + \sup_{0 \le t <1} (1-t)^{{-\theta_1}/{2}} \| \xi -
\mathbb{E} _t \xi\|.
\end{eqnarray*}
By assumption, we have $\| \xi- \mathbb{E}_t \xi\|^2 \le c
(1-t)^{\theta_1}$
hence the RHS is finite. From the lexicographical scale of the real
interpolation spaces (see \cite{bergh-lofstrom} or \cite
{bennet-sharp}), it follows
\[
(\LLL_2,{\mathbb{D}_{1,2}})_{\theta_1',2}
\supseteq(\mathit{L}_2,{\mathbb{D}_{1,2}})_{\theta_1,\infty}\qquad
\mbox{for all } \theta_1' \in\,]0,\theta_1[.
\]
Especially, $\| \xi- \mathbb{E}_t \xi\|^2 \le c (1-t)^{\theta_1}$
implies that
$ \sum_{n=0}^\infty n^{\theta_1'} \| I_n(f_n)\|^2 < \infty$ for all
$\theta_1' \in\,]0,\theta_1[$ (see \cite{GeissHujo}, Remark A.1).

(ii) In\vspace*{1.5pt} general (iv) $\nRightarrow$ (iii): let $(p_n)_{n=1}^\infty$ be an
ONB in $\LLL_2(\mu)$.
For simplicity, assume $T=1, m=1, f\equiv0$ and $\xi= \sum_{n=0}^\infty I_n(g_n \mathbh{1}_{]0,T]}^{\otimes n})$ so that
\[
Z_{s,x}= \sum_{n=1}^\infty n
I_{n-1}\bigl(g_n(x, \cdot) \mathbh{1} _{]0,s]}^{\otimes( n-1)}
\bigr).
\]
Setting $g_n := \beta_n (n!)^{-{1}/{2}} p_n^{\otimes n}$ we have
\[
\llVert Z_{s,\cdot}\rrVert _{\LLL_2(\mathbb{P}\otimes\mu
)}^2 = \sum
_{n=1}^\infty n \beta^2_n
s^{n-1}.
\]
For a sequence $(\beta_n)$ such that $\beta^2_1:=1$, $\beta^2_2:=0$,
$\beta^2_n:=\frac{1}{n ({\log(n-1)})^{2}}, n\ge3 $, we use Lemma A.1
of \cite{Seppaelae} which states that
\[
1 + \sum_{n=2}^\infty(\log n)^{-2}
s^n \sim_c \frac{1}{(1-s)(1-\log
(1-s))^2}
\]
(where for some $c \ge1$ and $A,B > 0$ the expression $A \sim_c B$ is
a short notation for $c^{-1} A \le B \le c A$). Hence
\[
\llVert Z_{s,\cdot}\rrVert _{\LLL_2(\mathbb{P}\otimes\mu
)}^2\sim
_c \frac{1}{(1-\log(1-s))^2}(1-s)^{-1},
\]
so that there does not exist any $\theta\in{]0,1]}$ for which property
(iii) holds.

But for any $h= \sum_{n=1}^\infty\alpha_n p_n$ such that $\|h\|
^2_{\LLL_2
(\mu)}= \sum_{n=1}^\infty\alpha_n^2 =1$ we have
\[
\biggl\llVert \int_\mathbb{R}
(Z_{t,x}-Z_{s,x} )h(x)\mu (\mathrm{d}x)\biggr\rrVert ^2
= \sum_{n=3}^\infty\alpha^2_n
\frac{1}{ ({\log( n-1)})^{2}} \bigl(t^{n-1} - s^{n-1}\bigr) \le
\frac{1}{ ({\log2})^{2}},
\]
which means that (iv) is fulfilled for any $\theta\in{]0,1]}$.
\end{rem}

We prepare some lemmas to prove Theorem \ref{mainthm}.

\begin{lem}\label{Hsmoothlemma}
Let $\eta\in\mathbb{H}\cap\mathbb{D}_{1,2}$ and $k\in\{1,\ldots
,m\}$.
Then for $\lambda$-a.e. $s,t\in{]r_{k-1},r_k[}$ with $s<t$ it holds
\[
\llVert \mathbb{E}_{t}\mathcal{D}_{t,\cdot}\eta-\mathbb
{E}_{s}\mathcal{D}_{s,\cdot}\eta\rrVert ^2_{\mathit{L}_2(\mathbb{P}\otimes\mu)}
\leq4 \int_s^t \frac{\|\mathbb
{E}_{r_k}\eta-\mathbb{E}
_r\eta\|^2}{(r_k-r)^2}\,\mathrm{d}r.
\]
\end{lem}
\begin{pf}
Let $\eta\in\mathbb{H}\cap\mathbb{D}_{1,2}$ be given by $\eta=
\sum_{n=0}^\infty I_n  ( \sum_{\alpha\in V_m^n}g_n^{\alpha}
\mathbh{1}_{\Lambda_\alpha}  ) $ where we assume that the functions
$f_n((t_1,x_1),\ldots,(t_n,x_n))$ are symmetric.
In the following, we use again the notation $\underline\alpha
:=(\alpha_2, \ldots,
\alpha_n)$. Since
\[
\mathcal{D}_{t,x}\eta= \sum_{n=1}^\infty
n I_{n-1} \biggl( \sum_{\alpha\in
V_m^n}g_n^{\alpha}(x,
\cdot) \mathbh{1}_{\Lambda_{\alpha_1}}(t )\mathbh{1}_{\Lambda_{\underline
\alpha}} \biggr)
\]
and since there exists a version of $\mathcal{D}\eta$ which is
constant on $
]r_{k-1},r_k[$ we get for
$s,t \in\,]r_{k-1},r_k[$ that
%
\begin{eqnarray}
&& \llVert \mathbb{E}_{t}\mathcal{D}_{t,\cdot}
\eta-\mathbb {E}_{s}\mathcal{D}_{s,\cdot}\eta\rrVert
^2_{\LLL_2(\mathbb{P}\otimes\mu)}
\nonumber
\\
\label{sumalph}
&&\quad= \Biggl\| \sum_{n=1}^\infty nI_{n-1}
\biggl(\sum_{\alpha\in V_m^n} g_n^{\alpha}
\mathbh{1}_{\Lambda_{\alpha_1}}(t) \mathbh{1}_{\Lambda_{\underline\alpha}} \bigl(\mathbh
{1}_{]0,t]}^{\otimes(n-1)}-\mathbh{1} _{]0,s]}^{\otimes(n-1)}
\bigr) \biggr) \Biggr\|^2_{\LLL_2(\mathbb{P}\otimes\mu)}
\\
\nonumber
&&\quad=\sum_{n=2}^\infty n n!
\mathop{\sum_{\alpha\in V_m^n}}_{\alpha_1=k}
\bigl\llVert g_n^{\alpha}\bigr\rrVert
^2_{\LLL_2(\mu^{\otimes n})} \lambda ^{\otimes(n-1)} \bigl(
\Lambda_{\underline\alpha}\cap\bigl({]0,t]}^{n-1} \setminus
{]0,s]}^{n-1}\bigr) \bigr).
\end{eqnarray}
For $\beta\in V_m^n$ and $1 \le l \le m$, we define
\[
\gamma_l(\beta):=\# \{i\mid\beta_i=l,\ i=1,\ldots,n\}.
\]

Notice that the intersection $ \Lambda_{\underline\alpha}\cap({]0,t]}^{n-1}
\setminus{]0,s]}^{n-1})$ is empty if $\sum_{d=k+1}^m \gamma
_{d}(\underline\alpha)
$ $>0$.
Therefore, setting
\[
\delta_{\underline\alpha}:= \mathbh{1}_{\{0\}} \Biggl( \sum
_{d=k+1}^m \gamma_{d}(\underline\alpha)
\Biggr)
\]
we have
\begin{eqnarray*}
&& \lambda^{\otimes(n-1)} \bigl(\Lambda_{\underline\alpha}\cap
\bigl({]0,t]}^{n-1} \setminus{]0,s]}^{n-1} \bigr) \bigr)
\\
&&\quad= \bigl( (t-r_{k-1} )^{\gamma_{k}(\underline\alpha
)}- (s-r_{k-1}
)^{\gamma_{k}(\underline\alpha)} \bigr)\prod_{1\le
l<k}
(r_l-r_{l-1} )^{\gamma_l
(\underline\alpha)} \delta_{\underline\alpha}.
\end{eqnarray*}
Using the symmetry of the functions in the chaos decomposition, we get that
\[
g_n^{\alpha}(x_1,\ldots,x_n)=g_n^{\pi(\alpha)}(x_{\pi(1)},
\ldots ,x_{\pi(n)})
\]
and hence $\llVert g_n^{\alpha}\rrVert ^2_{\LLL_2(\mu
^{\otimes
n})} = \|g_n^{\pi(\alpha)} \|^2_{\LLL_2(\mu^{\otimes
n})} $
for all $\pi\in\mathrm{S}(n)$ where we used the notation $\pi
(\alpha
):=(\alpha_{\pi(1)},\ldots,\alpha_{\pi(n)})$. Applying this fact, we
reduce our
summation over $\alpha\in V_m^n$ to a summation over equivalence
classes $[\alpha]\in V_m^n \slash\thicksim$ where
\[
\alpha\sim\beta\quad\Leftrightarrow\quad\exists\pi\in\mathrm{S}(n)\dvt \alpha =\pi(
\beta).
\]
Thus, since in (\ref{sumalph}) we fixed $\alpha_1$, by taking
equivalence classes for $V_m^{n-1}$ we obtain
%
\begin{eqnarray}
&& \llVert \mathbb{E}_{t}\mathcal{D}_{t,\cdot}
\eta-\mathbb {E}_{s}\mathcal{D}_{s,\cdot}\eta\rrVert
^2_{\LLL_2(\mathbb{P}\otimes\mu)}
\nonumber
\\
\label{normd}
&&\quad= \sum_{n=2}^\infty n n!\sum
_{[\underline\alpha]\in V_m^{n-1}
\slash
\thicksim}\frac{(n-1)!}{\gamma_1(\underline\alpha)!\cdots\gamma
_k(\underline\alpha)!} \bigl\llVert g_n^{(k,\underline\alpha)}
\bigr\rrVert ^2_{\LLL_2(\mu
^{\otimes n})}
\\
&& \qquad \hspace*{73pt}{}\times \bigl( (t-r_{k-1} )^{\gamma_{k}(\underline
\alpha)}- (s-r_{k-1}
)^{\gamma_{k}(\underline\alpha)} \bigr)\prod_{1
\le l<k}
(r_l-r_{l-1} )^{\gamma_l(\underline\alpha)}\delta_{\underline
\alpha},\nonumber
\end{eqnarray}
because the cardinality of the equivalence class $[\underline\alpha]$
is $\frac
{(n-1)!}{\gamma_1(\underline\alpha)!\cdots\gamma_k(\underline
\alpha)!}$ and $\gamma_l(\underline\alpha)$ is
invariant of permutations of $\underline\alpha$.
For $\gamma\ge1$, we estimate
\[
(t-r_{k-1} )^{\gamma}- (s-r_{k-1} )^{\gamma
}=
\int_s^t\gamma(r-r_{k-1})^{\gamma-1}\,\mathrm{d}r
\]
using for the integrand on the right-hand side the inequality
\[
(r -r_{k-1})^{\gamma-1} \le \frac{1}{(r_k-u)(u-r)} \int
_u^{r_k}\!\!\int_r^\rho
(v-r_{k-1})^{\gamma
-1} \,\mathrm{d}v\,\mathrm{d}\rho,
\qquad r_{k-1} \le r < u < r_k.
\]
For $u=\frac{r_k+r}{2}$ this leads to
\[
(t-r_{k-1} )^{\gamma}- (s-r_{k-1} )^{\gamma} \le
\frac{4}{(\gamma+1)}\int_s^t
\frac{ (r_k-r_{k-1})^{\gamma
+1}-(r-r_{k-1})^{\gamma+1} }{(r_k-r)^2}\,\mathrm{d}r.\vspace*{-12pt}
\]
\noindent This yields
\begin{eqnarray*}
&& \llVert \mathbb{E}_{t}\mathcal{D}_{t,\cdot}\eta-\mathbb
{E}_{s}\mathcal{D}_{s,\cdot}\eta\rrVert ^2_{\LLL_2(\mathbb{P}\otimes\mu)}
\nonumber
\\
&&\quad \le4\int_s^t\sum
_{n=1}^\infty n!\sum_{[\underline\alpha]\in
V_m^{n-1} \slash
\thicksim}
\frac{n!}{\gamma_1(\underline\alpha)!\cdots\gamma
_{k-1}(\underline\alpha)!(\gamma
_k(\underline\alpha)+1)!} \bigl\llVert g_n^{(k,\underline\alpha)}\bigr\rrVert
^2_{\LLL_2(\mu
^{\otimes n})}
\\
&& \hspace*{86pt}\qquad{}\times\frac{(r_k-r_{k-1})^{\gamma_k(\underline\alpha
)+1}-(r-r_{k-1})^{\gamma
_k(\underline\alpha)+1}}{(r_k-r)^2}\\
&&\qquad{}\times\prod_{1\le l<k}
(r_l-r_{l-1} )^{\gamma
_l(\underline\alpha)}\delta_{\underline\alpha}\,\mathrm{d}r,
\end{eqnarray*}
where for $\gamma_k(\underline\alpha)=0$ we have used
\[
0= (t-r_{k-1} )^{\gamma_{k}(\underline
\alpha)}- (s-r_{k-1} )^{\gamma_{k}(\underline\alpha)}
\le\int_s^t \frac
{(r_k-r_{k-1})-(r-r_{k-1})}{(r_k-r)^2}\,\mathrm{d}r.
\]
Because of
\[
\gamma_l(\alpha)=\gamma_l(\underline\alpha),\qquad  0<l<k
\quad\mbox{and} \quad\gamma_k(\alpha)=\gamma _k(\underline\alpha)+1
\]
if $\alpha= (k,\underline\alpha)$ we finally get
\begin{eqnarray*}
&& \llVert \mathbb{E}_{t}\mathcal{D}_{t,\cdot}\eta-\mathbb
{E}_{s}\mathcal{D}_{s,\cdot}\eta\rrVert ^2_{\LLL_2(\mathbb{P}\otimes\mu)}
\nonumber
\\
&&\quad \le4\int_s^t\sum
_{n=1}^\infty n!\sum_{[\alpha]\in V_m^n \slash
\thicksim}
\frac{n!}{\gamma_1(\alpha)!\cdots\gamma_k(\alpha)!} \bigl\llVert g_n^{\alpha}\bigr\rrVert
^2_{\LLL_2(\mu^{\otimes n})}
\\
&&\hspace*{80pt} \qquad{}\times\frac{(r_k-r_{k-1})^{\gamma_k(\alpha)}-(r-r_{k-1})^{\gamma
_k(\alpha
)}}{(r_k-r)^2}\prod_{l<k}
(r_l-r_{l-1} )^{\gamma_l(\alpha
)}\delta _{\alpha}\,\mathrm{d}r
\\
&&\quad= 4\int_s^t\sum
_{n=1}^\infty n!\sum_{\alpha\in V_m^n}
\bigl\llVert g_n^{\alpha}\bigr\rrVert ^2_{\LLL_2(\mu^{\otimes n})}
\frac
{\lambda
^{\otimes n} (\Lambda_\alpha
\cap ({]0,r_k]}^n \setminus{]0,r]}^n )
)}{(r_k-r)^2}\,\mathrm{d}r
\\
&&\quad=4\int_s^t\frac{\llVert \mathbb{E}_{r_k}\eta-\mathbb{E}_{r}\eta
\rrVert ^2}{(r_k-r)^2}\,\mathrm{d}r.
\end{eqnarray*}
\upqed\end{pf}

\begin{lem}\label{lemmaHsmoothcor}
If $\eta\in\mathbb{H}\cap\mathbb{D}_{1,2}$ and $k\in\{1,\ldots
,m\}$
then for $\lambda$-a.e. $t\in{]r_{k-1},r_k[}$
\[
\llVert \mathbb{E}_{t}\mathcal{D}_{t,\cdot}\eta\rrVert
^2_{\mathit{L}_2(\mathbb{P}\otimes
\mu)}\le\frac{\llVert \mathbb{E}_{r_k}\eta-\mathbb{E}_t\eta
\rrVert ^2}{r_k-t}.
\]
\end{lem}

\begin{pf}
Similar to the proof of the previous lemma, we get (using the same notation)
\begin{eqnarray*}
&& \llVert \mathbb{E}_{t}\mathcal{D}_{t,\cdot}\eta\rrVert
^2_{\LLL_2(\mathbb
{P}\otimes\mu)}
\nonumber
\\
&&\quad = \Biggl\| \sum_{n=1}^\infty nI_{n-1}
\biggl(\sum_{\alpha\in V_m^n} g_n^{\alpha}
\mathbh{1}_{\Lambda_{\alpha_1}}(t)\mathbh{1}_{\Lambda
_{\underline\alpha}}\mathbh{1}
_{]0,t]}^{\otimes(n-1)} \biggr) \Biggr\|^2_{\LLL_2(\mathbb{P}\otimes\mu)}
\nonumber
\\
&&\quad=\sum_{n=1}^\infty n n!\mathop{\sum_{\alpha\in V_m^n}}_{\alpha_1=k} \bigl\llVert g_n^{\alpha}\bigr\rrVert
^2_{\LLL_2(\mu^{\otimes n})} (t-r_{k-1} )^{\gamma_{k}(\underline\alpha)}\prod
_{l<k} (r_l-r_{l-1}
)^{\gamma_l
(\underline\alpha)}\delta_{\underline\alpha}
\\
&&\quad \le\sum_{n=1}^\infty n!\sum
_{\alpha\in V_m^n}\bigl\llVert g_n^{\alpha
}\bigr\rrVert
^2_{\LLL_2(\mu^{\otimes n})} \frac{(r_k-r_{k-1})^{\gamma_k(\alpha)} - (t-r_{k-1})^{\gamma
_k(\alpha)}
}{r_k-t}
\prod_{l<k}
(r_l-r_{l-1} )^{\gamma_l
( \alpha)} \delta_{\alpha}
\\
&&\quad= \frac{\llVert \mathbb{E}_{r_k}\eta-\mathbb{E}_t\eta\rrVert ^2}{r_k-t}.
\end{eqnarray*}
\upqed\end{pf}

\begin{lem}\label{savinglemma}
Suppose $u \in\,]r_{k-1},T]$, $\eta\in\mathbb{H}_u\cap\mathbb
{D}_{1,2}$ and $h \in\LLL_2(\mu)$.
Then the equation
\[
\frac{\mathbb{E}_{s} [\eta I_1 (\mathbh{1}_{]s,a]}h
) ]}{a-s}=\frac
{\mathbb{E}_{s} [\eta\int_{{]s,a]}\times\mathbb
{R}}h(x)M(\mathrm{d}t,\mathrm{d}x) ]}{a-s} =\int_\mathbb{R}
\mathbb{E}_{s}\mathcal{D}_{s,x} \eta h(x)\mu(\mathrm{d}x)
\]
is satisfied $\mathbb{P}$-a.s. for $\lambda$-a.e. $r_{k-1}< s<a \le r_k
\wedge u$.
\end{lem}

\begin{pf}
By the Clark--Ocone--Haussmann formula (\ref{COH}), we express $\eta$ as
\[
\eta=\mathbb{E}\eta+\int_{{]0,u]}\times\mathbb{R}}{}^p (\mathcal{D}
\eta )_{t,x} M(\mathrm{d}t,\mathrm{d}x).
\]
Thus we can write
\begin{eqnarray*}
&& \mathbb{E}_{s} \biggl[\eta\int_{{]s,a]}\times\mathbb
{R}}h(x)M(\mathrm{d}t,\mathrm{d}x)
\biggr]
\\
&&\quad= \mathbb{E}_{s} \biggl[\int_{{]0,u]}\times\mathbb{R}}{}^p
(\mathcal{D}\eta )_{t,x} M(\mathrm{d}t,\mathrm{d}x)\int_{{]s,a]}\times\mathbb{R}}h(x)M(\mathrm{d}t,\mathrm{d}x)
\biggr]
\end{eqnarray*}
(the constant $\mathbb{E}\eta$ multiplied with $\int_{{]s,a]}\times
\mathbb{R}
}h(x)M(\mathrm{d}t,\mathrm{d}x)$ gives zero when applying $\mathbb{E}_{s}$).
Using now the conditional It\^o-isometry, we arrive at
\begin{eqnarray*}
&&\mathbb{E}_{s} \biggl[\int_{{]0,u]}\times\mathbb{R}}{}^p
(\mathcal{D}\eta )_{t,x} M(\mathrm{d}t,\mathrm{d}x)\int_{{]s,a]}\times\mathbb{R}}h(x)M(\mathrm{d}t,\mathrm{d}x)
\biggr]
\\
&&\quad=\mathbb{E}_s\int_{{]s,a]}\times\mathbb{R}}\mathbb{E}_{t}
\mathcal {D}_{t,x}\eta h(x)\mathbh{m}(\mathrm{d}t,\mathrm{d}x)
\\
&&\quad=\int_{{]s,a]}\times\mathbb{R}}\mathbb{E}_{s}\mathcal
{D}_{t,x}\eta h(x)\mathbh{m}(\mathrm{d}t,\mathrm{d}x)
\\
&&\quad=(a-s)\int_\mathbb{R}\mathbb{E}_{s}
\mathcal{D}_{s,x}\eta h(x)\mu(\mathrm{d}x)
\end{eqnarray*}
since $\mathcal{D}\eta$ is $\mathbb{P}\otimes\mathbh{m}$-a.e.
constant on the interval
${]r_{k-1},r_k \wedge u[}$ with respect to the time variable because
$\eta$ is in $\mathbb{H}_u$.
\end{pf}

\begin{pf*}{Proof of Theorem~\ref{mainthm}}
In the following, we will indicate the dependency of the constants on
certain parameters but nevertheless the constants may vary from line to
line.
\begin{longlist}[(iii)  $\Rightarrow$ (ii):]
\item[(iii) $\Rightarrow$ (ii):] This step is analogous to the proof of \cite{GGG}, Theorem 1, $(C2_l)\Rightarrow (C3_l)$. It holds
\begin{eqnarray*}
\|Y_t - Y_s \|^2 &\le& 2(t-s)\int
_s^t \bigl\|f(r,X_r,Y_r,
\bar Z_r)\bigr\|^2 \,\mathrm{d}r+2 \int_s^t
\| Z_{r,\cdot} \|^2_{\LLL_2(\mathbb{P}\otimes\mu) }\,\mathrm{d}r
\\
&\le& c\bigl(L_f,\mu(\mathbb{R}),\kappa'\bigr) \int
_s^t \bigl(1 + \|Y_r
\|^2 + \|Z_{r,\cdot
} \|^2_{\LLL_2(\mathbb{P}\otimes\mu) } \bigr)\,\mathrm{d}r
\\
&\le& c\bigl(L_f,\mu(\mathbb{R}),\kappa',
c_3\bigr) \int_s^t
(r_k -r)^{\theta_k -1} \,\mathrm{d}r
\end{eqnarray*}
since $\int_0^T \|Y_r\|^2 \,\mathrm{d}r < \infty$ and $\|\bar Z_r\| \le\|
\kappa
'\|_{\LLL_2(\mu)}\|Z_{r,\cdot}\|_{\LLL_2(\mathbb
{P}\otimes\mu)}$.

\item[(ii) $\Rightarrow$ (i):]
The argument of \cite{GGG}, Theorem 1, $(C3_l)\Rightarrow(C4_l)$, works
here as well so that we have
\[
\llVert Y_{r_k}-\mathbb{E}_sY_{r_k}\rrVert
^2 \le4 \llVert Y_{r_k}-Y_s\rrVert
^2 \le\frac{4c_2}{\theta_k} (r_k-s)^{\theta_k}.
\]

\item[(i) $\Rightarrow$ (iii):]
Step 1.  We first assume that
%
\begin{equation}
\label{smoothness-assumption}
\xi\in\mathbb{H}\cap\mathbb{D}_{1,2} \mbox{ and } f
\mbox{ satisfies } (A_f 1).
\end{equation}
Because of the relation
%
\begin{equation}
\label{Y-relation}
Y_r =\mathbb{E}_r Y_{r_k} +
\mathbb{E}_r\int_r^{r_k}f(u,X_u,Y_u,
\bar Z_u )\,\mathrm{d}u,\qquad  r_{k-1}<r<r_k,
\end{equation}
Lemma \ref{ImkD} and Theorem \ref{diffthm}(iv) we have
$\mathbb{P}$-a.s. for $\mathbh{m}$-a.e. $(t,x) \in\,]r_{k-1},r_k[\,\times\,\mathbb
{R}$ that
%
\begin{eqnarray}
Z_{t,x} &=& \lim_{r\searrow t}
\mathcal{D}_{t,x}Y_r
\nonumber
\\
&=& \lim_{r\searrow t}\mathcal{D}_{t,x} \biggl(
\mathbb{E}_r Y_{r_k} + \mathbb{E}_r\int
_r^{r_k}f(u,X_u,Y_u,\bar
Z_u )\,\mathrm{d}u \biggr)
\nonumber
\\
\label{Z-limit}
&=& \lim_{r\searrow t} \biggl( \mathbb{E}_r
\mathcal{D}_{t,x} Y_{r_k} + \mathbb{E}_r
\mathcal{D}_{t,x}\int_r^{r_k}f(u,X_u,Y_u,
\bar Z_u )\,\mathrm{d}u \biggr)
\\
&=& \lim_{r\searrow t} \biggl( \mathbb{E}_r
\mathcal{D}_{t,x} Y_{r_k} + \mathbb{E}_r\int
_r^{r_k}\mathcal{D} _{t,x}f(u,X_u,Y_u,
\bar Z_u )\,\mathrm{d}u \biggr)
\nonumber
\\
&=& \mathbb{E}_t\mathcal{D}_{t,x} Y_{r_k} +
\mathbb{E}_t \int_t^{r_k}
\mathcal{D}_{t,x}f(u,X_u,Y_u,\bar
Z_u )\,\mathrm{d}u,
\nonumber
\end{eqnarray}
where we assumed the right continuous versions of the according
expressions: Since $Y_{r_k} \in\mathbb{H}\cap\mathbb{D}_{1,2}$
the expression $\mathcal{D}Y_{r_k}$ can be realized such that it is
constant in
$t $ on $ ]r_{k-1},r_k[$ and $ (\mathbb{E}_s\mathcal{D}_{t,x}
Y_{r_k})_{s \in\,]r_{k-1},r_k[}$
is a martingale (for fixed $x$).
From Lemma \ref{lemmaHsmoothcor}, we conclude that
%
\begin{equation}
\label{first-Z-estimate}
\| Z_{t,\cdot} \|_{\LLL_2(\mathbb{P}\otimes\mu)}
\le \frac{ \| Y_{r_k}- \mathbb{E}_t Y_{r_k} \| }{ \sqrt{ r_k-t}} + \int_t^{r_k}\bigl\|
\mathcal{D} _{t,\cdot}f(u,X_u,Y_u,\bar
Z_u )\bigr\|_{\LLL_2(\mathbb{P}\otimes
\mu)}\,\mathrm{d}u.
\end{equation}
Since Lemma \ref{composition}, the Lipschitz condition and relation
(\ref{levy-derivative}) imply
%
\[
\bigl|\mathcal{D}_{t,y}f(r,X_{r},Y_{r},\bar
Z_r )\bigr| \le L_f \bigl( \mathbh {1}_{[t,T]}(r) +
|\mathcal{D} _{t,y} Y_{r}|+ |\mathcal{D}_{t,y}
\bar Z_r |\bigr), \qquad y\neq0,
\]
and
\[
 \mathcal{D}_{t,0}f(r,X_{r},Y_{r},
\bar Z_r )
= \bigl( \mathbh{1} _{[t,T]}(r)\partial_x +
\mathcal{D}_{t,0} Y_{r} \partial_y+
\mathcal{D}_{t,0} \bar Z_r \partial_z
\bigr)f(r,X_{r},Y_{r},\bar Z_r ),
\]
we have
%
\begin{equation}
\label{Df-estimate}
\bigl\|\mathcal{D}_{t,\cdot}f(r,X_{r},Y_{r},
\bar Z_r )\bigr\|_{\mathit{L}_2(\mathbb{P}\otimes\mu)} \le  L_f \bigl(\sqrt{\mu(
\mathbb{R})} + \|\mathcal{D}_{t,\cdot} Y_{r}\|_{\LLL_2(\mathbb{P}\otimes
\mu)}
+ \|\mathcal{D}_{t,\cdot} \bar Z_r \|_{\LLL_2(\mathbb{P}\otimes\mu)}
\bigr).
\end{equation}

We take the Malliavin derivative of (\ref{Y-relation}), and by Lemmas~\ref{ImkD} and~\ref{lemmaHsmoothcor}, we get
%
\begin{equation}
\label{DY-estimate}
\|\mathcal{D}_{t,\cdot} Y_r
\|_{\mathit{L}_2(\mathbb
{P}\otimes\mu)}
\le \frac{ \| Y_{r_k}- \mathbb{E}_r Y_{r_k} \| }{ \sqrt{ r_k-r}} + \int_r^{r_k} \bigl\|
\mathcal{D}_{t,\cdot}f (u,X_u,Y_u,\bar
Z_u) \bigr\|_{\mathit{L}_2(\mathbb{P}\otimes\mu
)} \,\mathrm{d}u.
\end{equation}
In order to estimate $\|\mathcal{D}_{t,\cdot} \bar Z_r \|_{\mathit{L}_2(\mathbb{P}\otimes\mu
)}$, we will use the representation
\[
\bar Z_r = \frac{\mathbb{E}_r  [ (\mathbb{E}_u Y_{r_k})
I_1(\mathbh{1}_{]r,u]} \kappa
') ]}{u-r} + \int_r^{r_k}
\frac{\mathbb{E}_r [ f
(a,X_a,Y_a,\bar Z_a)
I_1(\mathbh{1}_{]r,a]} \kappa') ]}{a-r} \,\mathrm{d}a,
\]
for $\lambda$-a.e. $u$ such that $r_{k-1}<r<u \le r_k$,
which is a consequence of equation (\ref{Z-limit}), the fact that
$\mathbb{E}
_uY_{r_k} \in\mathbb{H}_u$, $f(a, X_a,Y_a,\bar Z_a ) \in\mathbb{H}_a$
and Lemma~\ref{savinglemma}.

Hence for $r_{k-1}<t<r<u < r_k$ the `conditional' H\"older inequality implies
\begin{eqnarray*}
&& \|\mathcal{D}_{t,\cdot} \bar Z_r \|_{\LLL_2(\mathbb
{P}\otimes\mu)}
\\
&&\quad \le \biggl\| \frac{\mathbb{E}_r  [ (\mathcal{D}_{t,\cdot
}(\mathbb{E}_u Y_{r_k}) ) I_1(\mathbh{1}
_{]r,u]} \kappa')  ]}{u-r} \biggr\|_{\LLL_2(\mathbb
{P}\otimes
\mu)}
\\
&& \qquad{}+ \int_r^{r_k} \biggl\|\frac{\mathbb{E}_r  [(\mathcal
{D}_{t,\cdot} f
(a,X_a,Y_a,\bar Z_a) ) I_1(\mathbh{1}_{]r,a]} \kappa') ]}{a-r} \biggr\|
_{\LLL_2(\mathbb{P}\otimes\mu)} \,\mathrm{d}a
\\
&&\quad \le\frac{ \| \mathcal{D}_{t,\cdot}(\mathbb{E}_u Y_{r_k}) \|
_{\LLL_2(\mathbb
{P}\otimes\mu)} \|\kappa' \|_{\LLL_2(\mu)} }{\sqrt{u-r}}
\\
&& \qquad{}+ c\bigl(L_f,\mu(\mathbb{R}),\kappa'\bigr) \int
_r^{r_k} \frac{1+\|
\mathcal{D}_{t,\cdot}Y_a\|
_{\LLL_2(\mathbb{P}\otimes\mu) }+\|\mathcal{D}_{t,\cdot
}\bar Z_a \|_{\LLL_2(\mathbb{P}\otimes\mu
)}}{\sqrt{a-r}}\,\mathrm{d}a,
\end{eqnarray*}
where we used that $ ( \mathbb{E}_r I_1(\mathbh{1}_{]r,u]} \kappa
')^2  )^{{1}/{2}}
\le c(\kappa')\sqrt{u-r}$ a.s., and from (\ref{Df-estimate}) one gets
the estimate for the
integral.
Choosing $u=\frac{r_k+r}{2}$, we conclude by Lemma~\ref{lemmaHsmoothcor} the inequality
\[
\frac{ \| \mathcal{D}_{t,\cdot} (\mathbb{E}_u Y_{r_k}) \|_{\mathit{L}_2(\mathbb{P}\otimes
\mu)} }{\sqrt{u-r}} \le2 \frac{ \| Y_{r_k}- \mathbb{E}_r Y_{r_k} \| }{ r_k-r}.
\]
From the estimate (\ref{DY-estimate}) for $ \mathcal{D}_{t,\cdot}
Y_r$ and the
above one for $ \mathcal{D}_{t,\cdot} \bar Z_r$, we obtain
\begin{eqnarray*}
&& \|\mathcal{D}_{t,\cdot} Y_r \|_{\LLL_2(\mathbb{P}\otimes
\mu)}+ \|
\mathcal{D} _{t,\cdot} \bar Z_r \|_{\LLL_2(\mathbb{P}\otimes\mu)}
\\
&&\quad\le c\bigl(\kappa'\bigr) \frac{ \| Y_{r_k}- \mathbb{E}_r Y_{r_k} \| }{ r_k-r}
\\
&& \qquad{}+ c\bigl(L_f,\mu(\mathbb{R}),\kappa'\bigr) \int
_r^{r_k}\frac{1+\|\mathcal
{D}_{t,\cdot}Y_a\|
_{\LLL_2(\mathbb{P}\otimes\mu) }+\|\mathcal{D}_{t,\cdot
}\bar Z_a \|_{\LLL_2(\mathbb{P}\otimes\mu
)}}{\sqrt{a-r}}\,\mathrm{d}a
\end{eqnarray*}
which can be treated using an iteration and Gronwall's lemma (see the
proof of Lemma 4 in \cite{GGG}) in order
to get
%
\begin{equation}
\label{DYandDbarZ-estimate}
\|\mathcal{D}_{t,\cdot} Y_r
\|_{\LLL_2(\mathbb{P}\otimes
\mu)}+ \|\mathcal{D}_{t,\cdot} \bar Z_r
\|_{\LLL_2(\mathbb{P}\otimes
\mu)}
\le c\bigl(L_f,\mu(\mathbb{R}),\kappa'\bigr)
\frac{ \| Y_{r_k}- \mathbb
{E}_r Y_{r_k} \| }{ r_k-r}.
\end{equation}
Hence from (\ref{first-Z-estimate}) and (\ref{Df-estimate}), it follows
%
\begin{eqnarray}
\| Z_{t,\cdot} \|_{\LLL_2(\mathbb{P}\otimes\mu)} &\le& \frac{ \| Y_{r_k}- \mathbb{E}_t Y_{r_k} \| }{ \sqrt{ r_k-t}}
\nonumber
\\[-8pt]
\label{Z-all-summands}
\\[-8pt]
\nonumber
&& {}+ c\bigl(L_f,\mu(\mathbb{R}),\kappa'\bigr) \int
_t^{r_k} \biggl(1+ \frac{
\| Y_{r_k}-
\mathbb{E}_r Y_{r_k} \| }{ r_k-r} \biggr)\,\mathrm{d}r.
\end{eqnarray}

Step 2. Here we use Remark \ref{approximationthm} and approximate $\xi
\in\mathbb{H}$ by a sequence $(\xi_n)_n \subseteq\mathbb{H} \cap
{\mathbb{D}_{1,2}}$
and $f$ such that $(A_f)$ is fulfilled by $(f_n)_n $ satisfying $(A_f
1)$. The convergence (\ref{Yn-Znconvergence}) implies that we can find
a subsequence $(Z^{n_k})$ which we will again denote by $(Z^n)$ such
that for $\lambda$-a.e. $t \in[0,T]$
%
\begin{eqnarray}
\label{aa-Z-convergence} \bigl\llVert Z^n_{t,\cdot} - Z_{t,\cdot}
\bigr\rrVert _{\LLL_2(\mathbb
{P}\otimes\mu)}^2 \to 0 \qquad \mbox{as } n \to\infty.
\end{eqnarray}
From the first step, we conclude that (\ref{Z-all-summands}) holds for
$Z^n$ and therefore
\begin{eqnarray*}
\| Z_{t,\cdot} \|_{\LLL_2(\mathbb{P}\otimes\mu)} & \le& \bigl\| Z_{t,\cdot}-
Z^n_{t,\cdot} \bigr\|_{\LLL_2(\mathbb{P}\otimes\mu)} + \bigl\| Z^n_{t,\cdot}
\bigr\|_{\LLL_2(\mathbb{P}\otimes\mu)}
\\
&\le& \bigl\| Z_{t,\cdot}- Z^n_{t,\cdot} \bigr\|_{\LLL_2(\mathbb
{P}\otimes\mu)} +
\frac{ \| Y^n_{r_k}- \mathbb{E}_t Y^n_{r_k} \| }{ \sqrt{ r_k-t}}
\nonumber
\\
&& {}+ c\bigl(L_f,\mu(\mathbb{R}),\kappa'\bigr) \int
_t^{r_k} \biggl( 1+ \frac{
\|
Y^n_{r_k}- \mathbb{E}_r Y^n_{r_k} \| }{ r_k-r} \biggr)\,\mathrm{d}r
\nonumber
\\
&\le& \frac{ \| Y_{r_k}- \mathbb{E}_t Y_{r_k} \| }{ \sqrt{ r_k-t}} + c\bigl(L_f,\mu (\mathbb{R}),
\kappa'\bigr) \int_t^{r_k} \biggl( 1
+ \frac{ \| Y_{r_k}-
\mathbb{E}_r Y_{r_k}
\| }{ r_k-r} \biggr) \,\mathrm{d}r
\nonumber
\\
&& {}+ \bigl\| Z_{t,\cdot}- Z^n_{t,\cdot} \bigr\|_{\LLL_2(\mathbb
{P}\otimes\mu)} +
\frac{
2\| Y_{r_k}- Y^n_{r_k} \| }{ \sqrt{ r_k-t}}
\nonumber
\\
&& {}+ c\bigl(L_f,\mu(\mathbb{R}),\kappa' \bigr) \int
_t^{r_k} \frac{ \|
Y_{r_k} - \mathbb{E}_r
Y_{r_k} -(Y^n_{r_k} - \mathbb{E}_r Y^n_{r_k} ) \| }{ r_k-r} \,\mathrm{d}r.
\end{eqnarray*}
For sufficiently large $n$ the terms in the second last line are
arbitrarily small. For the last term, we use the relation (\ref
{Y-relation}) and get
%
\begin{eqnarray}
&& \int_t^{r_k} \frac{ \| Y_{r_k} - \mathbb{E}_r Y_{r_k} -(Y^n_{r_k}
- \mathbb{E}_r
Y^n_{r_k} ) \| }{ r_k-r}
\,\mathrm{d}r
\nonumber
\\
&&\quad \le \int_t^{r_k} \frac{\int_r^{r_k} \| f(u,X_u,Y_u,\bar Z_u ) -
f(u,X_u,Y^n_u,\bar Z^n_u )\| \,\mathrm{d}u }{ r_k-r} \,\mathrm{d}r
\nonumber
\\
\label{arbitrarily-small}
&&\quad= \int_t^{r_k}\!\! \int_t^u
\frac{1 }{ r_k-r} \,\mathrm{d}r \bigl\| f(u,X_u,Y_u,\bar
Z_u ) - f\bigl(u,X_u,Y^n_u,
\bar Z^n_u \bigr)\bigr\| \,\mathrm{d}u
\\
&&\quad\le \biggl[\int_t^{r_k} \biggl( \int
_t^u \frac{1 }{ r_k-r} \,\mathrm{d}r
\biggr)^2 \,\mathrm{d}u \biggr]^{{1}/{2}}
\nonumber
\\
\nonumber
&& \qquad {}\times \biggl[ \int_t^{r_k} \bigl\|
f(u,X_u,Y_u,\bar Z_u ) - f
\bigl(u,X_u,Y^n_u,\bar Z^n_u
\bigr)\bigr\|^2 \,\mathrm{d}u \biggr]^{{1}/{2}},
\end{eqnarray}
where the last factor is arbitrarily small for sufficiently large $n$.

(i) $\Rightarrow$ (iv):
Step 1. We assume first that (\ref{smoothness-assumption}) holds for
$(\xi,f)$.
In the following, we use the notation $\mathbf{f}(r):=f
(r,X_r,Y_r,\bar Z_r)$. Then
equation (\ref{Z-limit}) allows us to write
\begin{eqnarray*}
&& \biggl\llVert \int_\mathbb{R} (Z_{t,x}-Z_{s,x}
)h(x)\mu (\mathrm{d}x)\biggr\rrVert
\\
&&\quad\le \biggl\|\int_\mathbb{R}(\mathbb{E}_t \mathcal
{D}_{t,x}Y_{r_k}- \mathbb{E}_s
\mathcal{D}_{s,x}Y_{r_k})h(x)\mu (\mathrm{d}x) \biggr\|
\\
&&\quad\quad {}+ \biggl\| \int_\mathbb{R} \biggl[\mathbb{E}_t\int
_t^{r_k} \mathcal{D}_{t,x}
\mathbf{f}(r)\,\mathrm{d}r- \mathbb{E}_s \int_s^{r_k}
\mathcal{D}_{s,x}\mathbf{f}(r)\,\mathrm{d}r \biggr]h(x)\mu (\mathrm{d}x) \biggr\|
\\
&&\quad \le \| \mathbb{E}_t \mathcal{D}_{t,\cdot}Y_{r_k}-
\mathbb {E}_s\mathcal{D}_{s,\cdot}Y_{r_k}
\|_{\LLL_2
(\mathbb{P}\otimes\mu)} \|h\|_{\LLL_2(\mu)}
\\
&&\qquad{}+ \biggl\| \int_\mathbb{R} \biggl[ \mathbb{E}_t \int
_t^{r_k}\mathcal{D}_{t,x}
\mathbf{f}(r)\,\mathrm{d}r- \mathbb{E}_s \mathbb{E}_t\int
_t^{r_k}\mathcal{D}_{t,x}\mathbf{f}(r)\,\mathrm{d}r
\biggr]h(x)\mu(\mathrm{d}x) \biggr\|
\\
&&\qquad {}+ \biggl\| \int_\mathbb{R}\mathbb{E}_s \int
_s^t \mathcal {D}_{s,x}
\mathbf{f}(r)\,\mathrm{d}r h(x)\mu(\mathrm{d}x) \biggr\|,
\end{eqnarray*}
where we have used that $\mathcal{D}\mathbf{f}(r) $ can be chosen to
be constant on
${]r_{k-1},r_k \wedge r[}$ that is,
we may exchange $\mathcal{D}_{s,x}\mathbf{f}(r)$ by $\mathcal
{D}_{t,x}\mathbf{f}(r)$ in the second
term.

From Lemma \ref{Hsmoothlemma}, we obtain that
\[
\| \mathbb{E}_t \mathcal{D}_{t,\cdot}Y_{r_k}-
\mathbb{E}_s\mathcal {D}_{s,\cdot}Y_{r_k}
\|^2_{\LLL_2(\mathbb{P}
\otimes\mu)} \le4 \int_s^t
\frac{ \|Y_{r_k}- \mathbb{E}_rY_{r_k}
\|^2}{
(r_k-r)^2} \,\mathrm{d}r.
\]
We combine (\ref{Df-estimate}) with (\ref{DYandDbarZ-estimate}) to get
%
\begin{equation}
\label{Df-estmated-again}
\bigl\| \mathcal{D}_{u, \cdot}\mathbf{f}(r) \bigr\|^2_{\LLL_2(\mathbb
{P}\otimes\mu)}
\le c\bigl( L_f, \mu(\mathbb{R} ), \kappa' \bigr)
\frac{ \|Y_{r_k}- \mathbb{E}_rY_{r_k} \|^2}{(r_k-r)^2},
\end{equation}
which is used to estimate
\begin{eqnarray*}
&& \biggl\| \int_\mathbb{R}\mathbb{E}_s\int
_s^t \mathcal {D}_{s,x}
\mathbf{f}(r)\,\mathrm{d}r h(x)\mu(\mathrm{d}x) \biggr\|\\
&&\quad\le \|h\|_{\LLL_2(\mu)} \int
_s^t \bigl\| \mathcal{D}_{s,\cdot
}
\mathbf{f}(r) \bigr\|_{\LLL_2(\mathbb{P}
\otimes\mu)} \,\mathrm{d}r
\\
&&\quad\le \|h\|_{\LLL_2(\mu)} \sqrt{c\bigl( L_f, \mu(\mathbb{R}),
\kappa'\bigr) } \int_s^t
\frac{ \| Y_{r_k}- \mathbb{E}_r Y_{r_k} \| }{ r_k-r}\,\mathrm{d}r.
\end{eqnarray*}

From Lemma \ref{savinglemma}, we conclude that
\[
\int_\mathbb{R}\mathbb{E}_t \mathcal{D}_{t,x}
\mathbf{f}(r) h(x)\mu(\mathrm{d}x) = \frac{ \mathbb{E}_t  [ I_1( \mathbh{1}
_{]t,r]}h)\mathbf{f}(r) ] }{r-t}.
\]
Applying the Clark--Ocone--Haussmann formula (\ref{COH}) and (\ref{Df-estmated-again}) yields
\begin{eqnarray*}
&& \biggl\| \int_\mathbb{R}\mathbb{E}_t
\mathcal{D}_{t,x} \mathbf {f}(r) h(x)\mu(\mathrm{d}x) - \mathbb{E}_s
\int_\mathbb{R}\mathbb{E}_t \mathcal{D}_{t,x}
\mathbf{f}(r) h(x)\mu(\mathrm{d}x) \biggr\|^2
\\[1pt]
&&\quad = \frac{1}{(r-t)^2} \biggl\| \int_{]s,t] \times\mathbb{R}} \phantom{)}^p
\bigl[ \mathcal{D}_{u,y} \mathbb{E}_t \bigl(
I_1( \mathbh {1}_{]t,r]}h) \mathbf{f}(r) \bigr)
\bigr]M(\mathrm{d}u ,\mathrm{d}y )\biggr\|^2
\\[1pt]
&&\quad \le \frac{1}{(r-t)^2}\int_s^t \!\!\int
_\mathbb{R}\mathbb{E} \bigl| \mathbb{E}_t \bigl[
I_1( \mathbh{1}_{]t,r]}h) \mathcal{D}_{u,y}
\mathbf{f}(r) \bigr] \bigr|^2 \mathbh{m}(\mathrm{d}u ,\mathrm{d}y )
\\[1pt]
&&\quad \le\frac{1}{r-t} \| h\|^2_{\LLL_2(\mu)} \int
_s^t \bigl\| \mathcal{D}_{u,\cdot}\mathbf{f}
(r)\bigr\|^2_{ \LLL_2(\mathbb{P}\otimes\mu)} \,\mathrm{d}u
\\[1pt]
&&\quad \le \frac{1}{r-t} \| h\|^2_{\LLL_2(\mu)} c\bigl(
L_f, \mu (\mathbb{R}), \kappa' \bigr) \int
_s^t \frac{ \| Y_{r_k}- \mathbb{E}_r Y_{r_k} \|^2 }{ ( r_k-r)^2}\,\mathrm{d}u.
\end{eqnarray*}
For the first inequality, we have used that for $u<t<r$ it holds
$\mathbb{P}\otimes\mathbh{m}$-a.e.
\[
\mathcal{D}_{u,y} \bigl[ I_1( \mathbh{1}_{]t,r]}h)
\mathbf{f}(r) \bigr] = I_1( \mathbh{1}_{]t,r]}h)
\mathcal{D}_{u,y}\mathbf{f}(r)
\]
since $ \mathcal{D}_{u,y} I_1( \mathbh{1}_{]t,r]}h)=0$. This can be
proved, for
example, applying \cite{GL}, Corollary 3.1,  and approximation.
Hence,
\begin{eqnarray*}
&& \Biggl\| \int_t^{r_k} \biggl[ \int
_\mathbb{R}\mathbb{E}_t \mathcal{D}_{t,x}
\mathbf{f}(r) h(x)\mu(\mathrm{d}x) - \mathbb{E}_s \int_\mathbb{R}
\mathbb{E}_t \mathcal {D}_{t,x}\mathbf{f}(r) h(x)\mu(\mathrm{d}x)
\biggr] \,\mathrm{d}r \Biggr\|
\\
&&\quad\le \| h\|_{\LLL_2(\mu)} \sqrt{c\bigl( L_f, \mu(\mathbb{R}),
\kappa' \bigr) } \int_t^{r_k}
\frac{ \| Y_{r_k}- \mathbb{E}_r Y_{r_k} \| }{ (r_k-r) \sqrt{r-t}}\,\mathrm{d}r \sqrt{t-s}.
\end{eqnarray*}
Consequently, we infer
%
\begin{eqnarray}
 &&\biggl\llVert \int_\mathbb{R}
(Z_{t,x}-Z_{s,x} )h(x)\mu (\mathrm{d}x)\biggr\rrVert
^2_{\LLL_2}
\nonumber
\\
&&\label{Z-difference-summary}\quad\le \| h\|^2_{\LLL_2(\mu)} c\bigl( L_f, \mu(
\mathbb{R}), \kappa ' \bigr) \biggl[\int_s^t
\frac{ \|Y_{r_k}- \mathbb{E}_rY_{r_k} \|^2}{ (r_k-r)^2} \,\mathrm{d}r
\\
&&\hspace*{105pt}\qquad{}+ \biggl( \int_t^{r_k} \frac{ \| Y_{r_k}- \mathbb{E}_r Y_{r_k} \|
}{ (r_k-r)
\sqrt{r-t}}\,\mathrm{d}r
\biggr)^2 (t-s) \biggr].\nonumber
\end{eqnarray}
Obviously (\ref{exp-of-Y}) implies
\[
\int_s^t \frac{ \|Y_{r_k}- \mathbb{E}_rY_{r_k} \|^2}{ (r_k-r)^2} \,\mathrm{d}r \le
c_1 \int_s^t(r_k-r)^{\theta_k-2}
\,\mathrm{d}r
\]
and
\begin{eqnarray*}
\biggl( \int_t^{r_k} \frac{ \| Y_{r_k}- \mathbb{E}_r Y_{r_k} \| }{ (r_k-r)
\sqrt{r-t}}\,\mathrm{d}r
\biggr)^2 &\le& c_1 \biggl( \int_0^1
(1-u)^{({\theta_k}/{2})-1} u^{ -{1}/{2}} \,\mathrm{d}u \biggr)^2
(r_k-t)^{\theta_k-1}
\\
&=& c_1 B^2 \biggl(\frac{\theta_k}{2}, \frac{1}{2}
\biggr) (r_k-t)^{\theta_k-1},
\end{eqnarray*}
where $B$ denotes the beta function.
For $\theta_k <1$ one can see that for all $s,t \in\,]r_{k-1}, r_k[$
with $s<t$ it holds
\begin{eqnarray*}
(r_k-t)^{\theta_k-1}(t-s) &\le& \frac{r_k-r_{k-1}}{1-\theta_k} \bigl(
(r_k-t)^{\theta_k-1} - (r_k-s)^{\theta_k-1} \bigr)
\\
& = & ( r_k-r_{k-1}) \int_s^t(r_k-r)^{\theta_k-2}
\,\mathrm{d}r
\end{eqnarray*}
since this inequality is equivalent to
\[
t-s:= \varepsilon(r_k-s) \le\frac{r_k-r_{k-1}}{1-\theta_k} \bigl[ 1- (1-
\varepsilon )^{1-\theta_k} \bigr]
\]
for $\varepsilon\in\,]0,1[$ and $s \in\,]r_{k-1}, r_k[$, and the last
inequality can be proved easily. For $\theta_k = 1$ we have
\[
(r_k-t)^0(t-s) \le \int_s^t
\frac{r_k-r_{k-1}}{r_k-r } \,\mathrm{d}r.
\]

Summarizing we get the assertion with
\[
c_4 = c_1 c\bigl( L_f, \mu(\mathbb{R}),
\kappa' \bigr) \biggl(1+ B^2 \biggl(\frac{\theta
_k}{2},
\frac{1}{2} \biggr) ( r_k-r_{k-1}) \biggr).
\]

Step 2.
Now we take the sequence $(\xi^n, f^n)_n$ from step 2 of the
implication $\mbox{(i)} \Rightarrow \mbox{(iii)}$
and proceed with (\ref{Z-difference-summary}) in the same way as we did
with (\ref{Z-all-summands}). In order to get the analogous estimate, we
use the relations
\begin{eqnarray*}
&&\int_s^t \frac{ \| \int_r^{r_k} f(u,X_u,Y_u,\bar Z_u ) -
f(u,X_u,Y^n_u,\bar Z^n_u ) \,\mathrm{d}u  \|^2 }{ (r_k-r)^2} \,\mathrm{d}r
\\
&&\quad\le \int_s^t \frac{ 1 }{ r_k-r} \,\mathrm{d}r \int
_{r_{k-1}}^{r_k} \bigl\| f(u,X_u,Y_u,
\bar Z_u ) - f\bigl(u,X_u,Y^n_u,
\bar Z^n_u \bigr) \bigr\|^2 \,\mathrm{d}u
\end{eqnarray*}
which is arbitrarily small for fixed $s, t\in[r_{k-1},r_k[ \setminus
N_k$ where $\lambda(N_k)=0$ and large $n \in\mathbb{N}$,
and
\begin{eqnarray*}
&& \int_t^{r_k} \frac{ \int_r^{r_k} \| f(u,X_u,Y_u,\bar Z_u ) -
f(u,X_u,Y^n_u,\bar Z^n_u)\| \,\mathrm{d}u }{ (r_k-r) \sqrt{r-t}}\,\mathrm{d}r
\\
&&\quad\le \frac{2 \int_t^{r_k} \| f(u,X_u,Y_u,\bar Z_u ) -
f(u,X_u,Y^n_u,\bar Z^n_u)\| \,\mathrm{d}u}{r_k-t} \int_t^{(r_k+t)/2}
\frac{ 1}{
\sqrt{r-t}}\,\mathrm{d}r
\\
&&\qquad{}+ \frac{\sqrt{2}}{ \sqrt{r_k-t}} \int_{(r_k+t)/2}^{r_k}
\frac{
\int_r^{r_k} \| f(u,X_u,Y_u,\bar Z_u ) - f(u,X_u,Y^n_u,\bar Z^n_u)\| \,\mathrm{d}u }{
r_k-r}\,\mathrm{d}r.
\end{eqnarray*}
\end{longlist}
For the last term, we can apply the estimate (\ref{arbitrarily-small})
to see that the RHS is arbitrarily small for large $n \in\mathbb{N}$.
\end{pf*}

\section{A sufficient condition on \texorpdfstring{$\xi$}{xi} for fractional smoothness}\label{sec6}

Assume ($A_f$) is satisfied for (\ref{beq2}) and $\xi\in\mathbb{H}$.
If $k=m$,
condition (i) of Theorem \ref{mainthm} means in fact
\[
\llVert \xi-\mathbb{E}_{s}\xi\rrVert ^2\leq
c_1(T-s)^{\theta_m},\qquad  s \in \,]r_{m-1},T].
\]
Following the ideas of \cite{GGG}, we will formulate a condition on
$\xi\in\mathbb{H}$ which implies that
(\ref{exp-of-Y}) of Theorem \ref{mainthm} holds for all $k \in\{
1,\ldots,m\}$.

Assume that $\check X$ and $X$ are processes on $(\Omega,\mathcal
{F},\mathbb{P})$ such that $\check X$ is an independent copy of
the L\'evy process $X$. We define for $0 \le t<r \le T$
%
\begin{equation}
\label{xtr}
X^{t,r}_s:=\int_0^s
\mathbh{1}_{{[0,T]}\setminus{]t,r]}}(u)\,\mathrm{d}X_u+\int_0^s
\mathbh{1} _{{]t,r]}}(u)\,\mathrm{d}\check X_u,\qquad s\in{[0,T]},
\end{equation}
that is, we obtain the process $X^{t,r}$ from $X$ by replacing it on
the interval ${]t,r]}$ by its independent copy.
Consequently, for the random measure $M^{t,r}$ w.r.t. $X^{t,r}$ we have
the relation
\[
M^{t,r}(B) = M \bigl(B\setminus \bigl({]t,r]}\times\mathbb{R} \bigr) \bigr) +
\check M \bigl(B\cap \bigl({]t,r]}\times\mathbb{R} \bigr) \bigr),\qquad B\in\mathcal {B}
\bigl({[0,T]}\times\mathbb{R}\bigr).
\]

By $ (\hat\mathcal{F}_t )_{t\in[0,T]}$ we denote the augmented natural
filtration w.r.t.~$(X,\check X)$ and put
$\LLL_2:= \LLL_2(\Omega$, $\hat\mathcal{F}_T,\mathbb
{P})$ (the notation $ (\mathcal{F}_t )_{t\in
[0,T]}$ we keep for the augmented natural filtration w.r.t.~$X$).

For symmetric $ f_n \in\LLL_2^n$ it holds
%
\begin{eqnarray}
\label{chaos-comperation}
\bigl\|I^{t,r}_n(f_n)-I_n(f_n)
\bigr\|^2=2n!\bigl\llVert f_n (1-\mathbh{1}_{
 (
([0,T]\setminus]t,r] )\times\mathbb{R} )^n} )
\bigr\rrVert ^2_{\LLL_2^n}.
\end{eqnarray}

For any $\eta\in\LLL_2$ given by $\eta=\sum_{n=0}^\infty I_n (f_n)$,
we define
\[
\eta^{t,r}:=\sum_{n=1}^\infty
I^{t,r}_n (f_n).
\]
%
\begin{thm}\label{suffcond}
Assume that $\xi\in\mathbb{H}$ and ($A_f$) is satisfied for (\ref
{beq2}). If there exist constants $c>0$ and $\theta_k \in\,]0,1]$ such that
\[
\bigl\|\xi-\xi^{t,r_k} \bigr\|^2\leq c(r_k-t)^{\theta_k}
\qquad\mbox{for all } t\in [r_{k-1},r_k]
\]
then
\[
\|Y_{r_k}-\mathbb{E}_tY_{r_k}\|^2
\le C(r_k-t)^{\theta_k} \qquad\mbox{for all } t\in [r_{k-1},r_k].
\]
\end{thm}

\begin{rem}\label{fracsmooth}
(i)
For $f=0$, it follows from Theorem \ref{suffcond} the implication
%
\begin{eqnarray}
&&\bigl\|\xi-\xi^{t,r_k} \bigr\|^2\leq
c(r_k-t)^{\theta_k}\qquad \mbox{for all } t\in [r_{k-1},r_k]
\nonumber
\\[-8pt]
\label{sufficient}
\\[-8pt]
\nonumber
&&\quad\Longrightarrow \quad \| \mathbb{E}_{r_k} \xi- \mathbb{E}_t\xi
\|^2\leq c(r_k-t)^{\theta_k}\qquad  \mbox{for all } t
\in[r_{k-1},r_k].
\end{eqnarray}
For certain $\xi$ the implication (\ref{sufficient}) is in fact an
equivalence: for example, if $\xi=g(X_{r_m}-X_{r_{m-1}},\ldots,
X_{r_1}-X_{r_0}) \in\LLL_2$
such that $g$ is a symmetric function and $r_k = \frac{kT}{m}$, for
$k=0,\ldots,m$. A more detailed discussion under which conditions
equivalence holds for (\ref{sufficient})
as well as an example where $ \| \mathbb{E}_{r_k} \xi- \mathbb
{E}_t\xi\|^2\leq
c(r_k-t)^{\theta_k}, \mbox{ for all } t\in[r_{k-1},r_k]$ does not
imply $\|\xi-\xi^{t,r_k} \|^2\leq c(r_k-t)^{\theta_k}, \mbox{ for
all }
t\in[r_{k-1},r_k] $ can be found in \cite{GeissLect}.
\begin{longlist}[(ii)]
\item[(ii)]
If $\xi\in\mathbb{H}$ the case $\Theta=(1,1,\ldots,1)$ corresponds to
Malliavin differentiability:
%
\begin{eqnarray}
&& \exists c>0\dvt  \bigl\|\xi-\xi^{t,r_k} \bigr\|^2\leq
c(r_k-t) \qquad \mbox{for all } t \in[r_{k-1},r_k],
k=1,\ldots,m
\nonumber
\\[-8pt]
\label{theta=1relation}\\[-8pt]
&&\hspace*{94pt}\nonumber\qquad\qquad\qquad\qquad\iff \quad \xi\in{\mathbb{D}_{1,2}}.
\end{eqnarray}
Indeed, using the notation of the proof of Lemma \ref{Hsmoothlemma} and
setting ${n\choose\gamma(\alpha)}:=\frac{n!}{\gamma_1(\alpha
)!\cdots
\gamma_m(\alpha)!}$ we have
\begin{eqnarray*}
\bigl\|\xi-\xi^{t,r_k} \bigr\|^2 & = & 2 \sum
_{n=1}^\infty n!\sum_{[\alpha]\in V_m^n \slash\thicksim
}
\pmatrix{n
\cr
\gamma(\alpha)} \bigl\llVert g_n^{\alpha}
\bigr\rrVert ^2_{\LLL_2(\mu^{\otimes n})}
\\
&& \hspace*{68pt}{}\times \bigl( (r_k-r_{k-1})^{\gamma_k(\alpha
)}-(t-r_{k-1})^{\gamma_k(\alpha)}
\bigr) \mathop{\prod_{1\le l\le m}}_{l \neq k}
(r_l-r_{l-1} )^{\gamma_l(\alpha)}.
\end{eqnarray*}
This implies for $r:=\frac{t-r_{k-1}}{r_k-r_{k-1}}$ and $R:=\max_{1
\le
k \le m} \frac{1}{r_k-r_{k-1}}$ that
\begin{eqnarray*}
\frac{ \|\xi-\xi^{t,r_k} \|^2}{ r_k-t } &= & \frac{2}{r_k-r_{k-1}} \sum_{n=1}^\infty
n!\sum_{[\alpha]\in V_m^n
\slash\thicksim}\pmatrix{n
\cr
\gamma(\alpha)} \bigl
\llVert g_n^{\alpha}\bigr\rrVert ^2_{\LLL_2(\mu^{\otimes n})}
\lambda ^n(\Lambda_\alpha)
\\
&& \hspace*{103pt}{}\times\mathbh{1}_{ \{ \gamma_k(\alpha)\ge1\}} \bigl( 1+r + \cdots+ r ^{\gamma
_k(\alpha)-1} \bigr)
\\
&\le& 2 R \sum_{n=1}^\infty n!\sum
_{[\alpha]\in V_m^n \slash
\thicksim} \pmatrix{n
\cr
\gamma(\alpha)} \bigl\llVert
g_n^{\alpha}\bigr\rrVert ^2_{\LLL_2(\mu^{\otimes n})} \lambda
^n(\Lambda_\alpha) \gamma_k(\alpha)
\\
&\le& 2 R \| \mathcal{D}\xi\|^2_{\mathbb{P}\otimes\mathbh{m}}
\end{eqnarray*}
since $\gamma_k(\alpha) \le n$.
On the other hand, we get because of $n= \sum_{k=1}^m \gamma_k(\alpha)$
for $\alpha\in V_m^n$ from the above relation
that
\begin{eqnarray*}
\| \mathcal{D}\xi\|^2_{\mathbb{P}\otimes\mathbh{m}} &=& \sum
_{k=1}^m \sum_{n=1}^\infty
n!\sum_{[\alpha]\in V_m^n \slash
\thicksim
}\pmatrix{n
\cr
\gamma(\alpha)} \bigl
\llVert g_n^{\alpha}\bigr\rrVert ^2_{\LLL_2(\mu^{\otimes n})}
\lambda ^n(\Lambda_\alpha) \gamma_k(\alpha)
\\
&\le& \frac{T}{2} \sup_{1 \le k \le m} \sup_{r_{k-1} < t <r_k}
\frac{
\|\xi-\xi^{t,r_k} \|^2}{ r_k-t }.
\end{eqnarray*}
In \cite{GeissLect}, there is an example which shows that (\ref
{theta=1relation}) is not necessarily true without assuming $\xi\in
\mathbb{H}$.
\end{longlist}
\end{rem}

\begin{example}\label{ex63}
If $X$ is any square integrable L\'evy process it
holds for $\xi:= \mathbh{1}_{]K, \infty[}(X_1)$ with $K \in\mathbb{R}$
that
\[
\xi\in{\mathbb{D}_{1,2}} \quad \iff \quad \sigma=0\quad \mbox{and}\quad \int
_{\{|x|
\le1\}} |x| \,\mathrm{d}\nu (x) < \infty
\]
(see \cite{Laukkarinen}, Example 3.1).
If $X$ is a tempered $\alpha$-stable process given by the L\'evy measure
\[
\nu_\alpha(\mathrm{d}x) = \frac{d}{|x|^{1+\alpha}} (1+|x|)^{-m} \mathbh
{1}_{\{x\neq
0\}} \,\mathrm{d}x,
\]
where $d>0$, $\alpha\in\,]0,2[$ and $m \in\,]2-\alpha, \infty[$, it follows
from \cite{GGL}, Section~4.2, that
\[
\xi\in\mathbb{B}_{2,\infty}^{{1}/{2}} := (\LLL_2,{
\mathbb {D}_{1,2}})_{{1}/{2},\infty},
\]
that is, (see Remarks \ref{frac-smooth-comment}(i)
and \ref{fracsmooth}(i)) there exists a $c>0$:
\[
\bigl\|\xi-\xi^{t,1} \bigr\|^2\leq c(1-t)^{{1}/{2}} \qquad \mbox{for
all } t\in[0,1].
\]
Consequently, for any $\alpha\in[1,2[$ the above $\xi$ is in
$\mathbb
{B}_{2,\infty}^{{1}/{2}}$ but not in ${\mathbb{D}_{1,2}}$.
\end{example}

\begin{pf*}{Proof of Theorem \ref{suffcond}}
If $(Y,Z)$ is a solution of (\ref{beq2}), then $(Y^{t,r},Z^{t,r})$ solves
\[
\mathcal{Y}_u=\xi^{t,r}+\int_u^T
f \bigl(s,X^{t,r}_{s}, \mathcal{Y}_{s}, \bar
\mathcal{Z}_s \bigr)\,\mathrm{d}s-\int_{{]u,T]}\times\mathbb{R}}
\mathcal{Z}_{s,x}M^{t,r}(\mathrm{d}s,\mathrm{d}x).
\]
From (\ref{chaos-comperation}), we conclude that
\[
\bigl\| \mathbb{E}_r I^{t,r}_n(f_n)-
\mathbb{E}_r I_n(f_n)\bigr\|^2=2 \bigl\|
\mathbb {E}_t I_n(f_n) -
\mathbb{E}_r I_n(f_n) \bigr\|^2.
\]
Since $Y_{r_k}$ is $\mathcal{F}_{r_k}$-measurable this implies for $t
\in\,]r_{k-1},r_k[ $ that
\[
2 \|Y_{r_k}-\mathbb{E}_{t}Y_{r_k}
\|^2 = \bigl\|Y_{r_k}-Y^{t,r_k}_{r_k}
\bigr\|^2.
\]
Since $M$ and $M^{t,r_k}$ coincide on ${]r_k,T]}\times\mathbb{R}$ we have
\begin{eqnarray*}
Y_{r_k}-Y^{t,r_k}_{r_k} & = & \xi-\xi^{t,r_k}
\\
&&{}+ \int_{r_k}^T f (s,X_{s},
Y_{s},\bar Z_s)- f\bigl(s,X^{t,r_k}_{s},
Y^{r_k}_{s},\bar Z^{t,r_k}_s\bigr)\,\mathrm{d}s
\\
&&{}-\int_{{]r_k,T]}\times\mathbb{R}} \bigl(Z_{s,x}-Z^{t,r_k}_{s,x}
\bigr) M(\mathrm{d}s,\mathrm{d}x).
\end{eqnarray*}
By Theorem \ref{continuitythm}, we get
\begin{eqnarray*}
&&\mathbb{E}\bigl\llvert Y_{r_k}-Y^{t,r_k}_{r_k}\bigr
\rrvert ^2+\mathbb{E}\int_{]r_k,T] \times
\mathbb{R}} \bigl\llvert
Z_{s,x}-Z^{t,r_k}_{s,x}\bigr\rrvert ^2
\mathbh{m}(\mathrm{d}s,\mathrm{d}x)
\\
&&\quad\le C\biggl(\mathbb{E}\bigl\llvert \xi-\xi^{t,r_k}\bigr\rrvert
^2
+\mathbb{E}\int_{r_k}^T \bigl| f
(s,X_{s}, Y_{s},\bar Z_s)- f
\bigl(s,X^{t,r_k}_{s}, Y_{s},\bar Z_s
\bigr) \bigr|^2\,\mathrm{d}s\biggr),
\end{eqnarray*}
which can be reduced by the Lipschitz property of $f$ to
\begin{eqnarray*}
&&\mathbb{E}\bigl\llvert Y_{r_k}-Y^{t,r_k}_{r_k}\bigr
\rrvert ^2+\mathbb{E}\int_{{]r_k,T]}\times\mathbb{R}} \bigl\llvert
Z_{s,x}-Z^{t,r_k}_{s,x}\bigr\rrvert ^2
\mathbh{m}(\mathrm{d}s,\mathrm{d}x)
\\
&&\quad\leq C \biggl(\mathbb{E}\bigl\llvert \xi-\xi^{t,r_k}\bigr\rrvert
^2+\mathbb {E}\int_{r_k}^T
L_f^2\bigl\llvert X_s-X^{t,r_k}_{s}
\bigr\rrvert ^2\,\mathrm{d}s \biggr).
\end{eqnarray*}
By definition of $X^{t,r_k}$ in (\ref{xtr}), we get for $s> r_k$
\[
\mathbb{E}\bigl\llvert X_s-X^{t,r_k}_s\bigr\rrvert ^2=\mathbb{E}\bigl\llvert X_{r_k}- X_t
+ (\check X_{r_k} -\check X_t )\bigr\rrvert ^2
=C_1(r_k-t).
\]
Thus, there is a constant $\tilde C$ such that
\begin{eqnarray*}
&&\mathbb{E}\bigl\llvert Y_{r_k}-Y^{t,r_k}_{r_k}\bigr
\rrvert ^2+\mathbb{E}\int_{{]r_k,T]}\times\mathbb{R}} \bigl\llvert
Z_{s,x}-Z^{t,r_k}_{s,x}\bigr\rrvert ^2
\mathbh{m}(\mathrm{d}s,\mathrm{d}x)
\\
&&\quad\leq C\mathbb{E}\bigl\llvert \xi-\xi^{t,r_k}\bigr\rrvert ^2+
\tilde C (r_k-t),
\end{eqnarray*}
which implies the assertion.
\end{pf*}
%

\section{Concluding remarks}\label{sec7}
\begin{enumerate}[4.]
\item[1.] The assumption that the L\'evy process $X$ is square integrable
could be avoided by using a more general formulation
of the Clark--Ocone--Haussmann formula and modifying the dependency of
the generator $f(t,X_t,Y_t, \bar Z_t)$ on $X_t$. (If $X$ is not square
integrable,
$X_t$ does not belong to ${\mathbb{D}_{1,2}}$.)
\item[2.] A generalization to the setting of a $d$-dimensional L\'evy
process seems to be possible as well and similar results might be
expected. For example, for a multidimensional L\'evy process without
Brownian part, a chaos decomposition and a Clark--Ocone--Haussmann
formula can be found
in \cite{LastPenrose} and \cite{LastPenrose2}. This could be extended
to general L\'evy processes.
\item[3.] In this paper, the dependency of the driver with respect to the
$Z$ process is given by the integral $\int Z_{t,x}\kappa(\mathrm{d}x)$. A
generalization to the dependency
on finitely many integrals,
\[
f \biggl(s,X_s,Y_s,\int Z_{t,x}
\kappa_1(\mathrm{d}x),\ldots,\int Z_{t,x}\kappa _n(\mathrm{d}x)
\biggr),
\]
where the variables $z_1,\ldots,z_n$ in the generator underly the same
assumptions as for one $z$-variable appears to be straightforward. Note
that the choice
$\kappa=\delta_0$ covers the case for the $Z$-variable from \cite{bbp}, for instance.
\item[4.] The investigation of the case where the terminal condition or the
generator depends on paths of a process of a L\'evy driven SDE
is of major interest for further research, as well as the extension to
assumptions beyond the Lipschitz generator setting like quadratic
drivers.
\end{enumerate}

\section*{Acknowledgement}
We would like to thank S. Geiss for fruitful discussions and valuable
suggestions and the unknown referee for his critical remarks.

Alexander Steinicke was partially supported by the project 133914
\textit{Stochastic and Harmonic Analysis, Interactions and Applications}
of the Academy of Finland.






\printhistory
\end{document}